\documentclass{gtpart}     
\usepackage{pinlabel}  
\usepackage{graphicx}  
\usepackage[all]{xy}
\usepackage{amscd}
\title{Charts, signatures, and stabilizations of Lefschetz fibrations}

\author{Hisaaki Endo}
\givenname{Hisaaki}
\surname{Endo}
\address{Department of Mathematics\\Tokyo Institute of Technology \\\newline 
2-12-1 Oh-okayama\\Meguro-ku\\Tokyo 152-8551\\Japan}
\email{endo@math.titech.ac.jp}
\urladdr{}

\author{Isao Hasegawa}
\givenname{Isao}
\surname{Hasegawa}
\address{Ministry of Health\\Labour and Welfare \\\newline 
1-2-2 Kasumigaseki\\Chiyoda-ku\\Tokyo 100-8916\\Japan}
\email{hasegawa-isao01@mhlw.go.jp}

\author{Seiichi Kamada}
\givenname{Seiichi}
\surname{Kamada}
\address{Department of Mathematics\\Osaka City University \\\newline 
3-3-138 Sugimoto\\Sumiyoshi-ku\\Osaka 558-8585\\Japan}
\email{skamada@sci.osaka-cu.ac.jp}

\author{Kokoro Tanaka}
\givenname{Kokoro}
\surname{Tanaka}
\address{Department of Mathematics\\Tokyo Gakugei University \\\newline 
4-1-1 Nukuikita-machi\\Koganei-shi\\Tokyo 184-8501\\Japan}
\email{kotanaka@u-gakugei.ac.jp}

\arxivreference{}
\arxivpassword{}

\volumenumber{}
\issuenumber{}
\publicationyear{}
\papernumber{}
\startpage{}
\endpage{}
\doi{}
\MR{}
\Zbl{}
\received{}
\revised{}
\accepted{}
\published{}
\publishedonline{}
\proposed{}
\seconded{}
\corresponding{}
\editor{}
\version{}

\newtheorem{thm}{Theorem}[section]    
\newtheorem{prop}[thm]{Proposition}          
\theoremstyle{definition}
\newtheorem{defn}[thm]{Definition}    
\newtheorem{prob}[thm]{Problem}          
\newtheorem{exam}[thm]{Example}
\newtheorem{rem}{Remark}             
\def\co{\colon\thinspace}

\begin{document}

\begin{abstract}    
We employ a certain labeled finite graph, called a chart, in a closed oriented surface 
for describing the monodromy of a(n achiral) Lefschetz fibration over the surface. 
Applying charts and their moves with respect to Wajnryb's presentation of mapping class groups, 
we first generalize a signature formula for Lefschetz fibrations over the $2$--sphere 
obtained by Endo and Nagami to that for Lefschetz fibrations over 
arbitrary closed oriented surface. 
We then show two theorems on stabilization of Lefschetz fibrations under fiber summing 
with copies of a typical Lefschetz fibration as generalizations of a theorem of Auroux. 
\end{abstract}

\maketitle



\section{Introduction}


Matsumoto \cite{Matsumoto1986} 
proved that every Lefschetz fibration of genus one over a closed oriented surface 
is isomorphic to a fiber sum of copies of a holomorphic elliptic fibration on 
$\mathbb{CP}^2\# 9\overline{\mathbb{CP}}^2$ and a trivial torus bundle over the surface 
if it has at least one critical point. 
This result played a crucial role in completing the classification of diffeomorphism types 
of elliptic surfaces (see Gompf and Stipsicz \cite[Section 8.3]{GS1999}). 
Although such a classification has not been established for Lefschetz fibrations of higher 
genus, Auroux \cite{Auroux2003} proved a stabilization theorem for Lefschetz fibrations 
of genus two, which states that every Lefschetz fibration of genus two over the $2$--sphere 
becomes isomorphic to a fiber sum of copies of three typical fibrations after fiber summing with 
a holomorphic fibration on $\mathbb{CP}^2\# 13\overline{\mathbb{CP}}^2$. 
Auroux \cite{Auroux2005} gave a generalization of this theorem for Lefschetz fibrations 
of higher genus, which states that two Lefschetz fibrations of the same genus over the 
$2$--sphere which have the same signature, the same numbers of singular fibers of each type, 
and admit sections of the same self-intersection number become isomorphic after 
fiber summing the same number of copies of a `universal' Lefschetz fibration. 

Kamada \cite{Kamada1992, Kamada1996} introduced charts, 
which are labeled finite graphs in a disk, to describe monodromies of surface braids 
(see also a textbook \cite{Kamada2002} of Kamada). 
Kamada, Matsumoto, Matumoto, and Waki \cite{KMMW2005} considered a variant of chart 
for Lefschetz fibrations of genus one 
to give a remarkably simple proof of the above result of Matsumoto. 
Furthermore Kamada \cite{Kamada2012}, and Endo and Kamada \cite{EK2013, EK2014} 
made use of generalized charts 
to give a simple proof of the above theorem of Auroux for Lefschetz 
fibrations of genus two, and to investigate a stabilization theorem and an invariant 
for hyperelliptic Lefschetz fibrations of arbitrary genus. 
See also Baykur and Kamada \cite{BK2010}, and Hayano \cite{Hayano2011} 
for applications of charts to broken Lefschetz fibrations. 

In this paper we introduce a chart description for Lefschetz fibrations of genus greater 
than two over closed oriented surfaces of arbitrary genus to show a signature formula 
and two theorems on stabilization for such fibrations. In Section 2 
we introduce charts and chart moves with respect to Wajnryb's presentation 
of mapping class groups to examine monodromies of Lefschetz fibrations. 
After a short survey of Meyer's signature cocycle, we generalize a signature formula 
\cite{EN2004} of Endo and Nagami for Lefschetz fibrations over the $2$--sphere to that 
for Lefschetz fibrations over a closed oriented surface of arbitrary genus in Section 3. 
Section 4 is devoted to proofs of two theorems on stabilization of Lefschetz fibrations 
under fiber summing with copies of a `universal' Lefschetz fibration. 
In particular the first of our stabilization theorems is a generalization of the theorem 
of Auroux \cite{Auroux2005}. 
We make several comments on variations of chart description 
and propose some possible directions for future research in Section 5. 


\section{Chart description for Lefschetz fibrations}


In this section we review a definition and properties of Lefschetz fibrations 
and introduce a chart description for Lefschetz fibrations of genus greater than two. 

\subsection{Lefschetz fibrations and their monodromies}

In this subsection we review a precise definition and basic properties of Lefschetz fibrations. 
More details can be found in 
Matsumoto \cite{Matsumoto1996} and Gompf and Stipsicz \cite{GS1999}. 

Let $\Sigma_g$ be a connected closed oriented surface of genus $g$. 

\begin{defn}\label{LF}
Let $M$ and $B$ be connected closed oriented smooth $4$--manifold and $2$--manifold, 
respectively. 
A smooth map $f\co M\rightarrow B$ is called a {\it Lefschetz fibration} of 
genus $g$ if it satisfies the following conditions: 

(i) the set $\Delta\subset B$ of critical values of $f$ is finite 
and $f$ is a smooth fiber bundle 
over $B-\Delta$ with fiber $\Sigma_g$; 

(ii) for each $b\in\Delta$, there exists a unique 
critical point $p$ in the {\it singular fiber} $F_b:=f^{-1}(b)$ 
such that $f$ is locally written as 
$f(z_1,z_2)=z_1z_2$ or $\bar{z}_1z_2$ with respect to some local complex 
coordinates around $p$ and $b$ which are compatible with 
orientations of $M$ and $B$; 

(iii) no fiber contains a $(\pm 1)$--sphere. 

We call $M$ the {\it total space}, $B$ the {\it base space}, and $f$ the {\it projection}. 
We call $p$ a critical point of {\it positive type} (resp. of {\it negative type}) 
and $F_b$ a singular fiber of {\it positive type} (resp. of {\it negative type}) if $f$ is locally 
written as $f(z_1,z_2)=z_1z_2$ (resp. $f(z_1,z_2)=\bar{z}_1z_2$) in (ii). 
For a regular value $b\in B$ of $f$, $f^{-1}(b)$ is often called a {\it general fiber}. 
\end{defn}

\begin{rem}
A Lefschetz fibration in this paper is called an {\it achiral} Lefschetz fibration 
in many other papers. 
\end{rem}

Let $f\co M\rightarrow B$ and $f'\co M'\rightarrow B$ be Lefschetz fibrations of genus $g$ 
over the same base space $B$. We say that $f$ is {\it isomorphic} to $f'$ 
if there exist orientation preserving diffeomorphisms $H\co M\rightarrow M'$ 
and $h\co B\rightarrow B$ which satisfy $f'\circ H=h\circ f$. 
If we can choose such an $h$ isotopic to the identity relative to a given base point 
$b_0\in B$, we say that $f$ is {\it strictly isomorphic} to $f'$. 

Let $\mathcal{M}_g$ be the mapping class group of $\Sigma_g$, 
namely the group of all isotopy classes of orientation preserving 
diffeomorphisms of $\Sigma_g$. 
We assume that $\mathcal{M}_g$ acts on the {\it right}: 
the symbol $\varphi\psi$ means that we apply $\varphi$ first 
and then $\psi$ for $\varphi, \psi \in \mathcal{M}_g$. 
We denote the mapping class group of $\Sigma_g$ acting on the left by $\mathcal{M}_g^*$. 
Hence the identity map $\mathcal{M}_g\rightarrow \mathcal{M}_g^*$ is an anti-isomorphism. 

Let $f\co M\rightarrow B$ be a Lefschetz fibration of genus $g$ 
as in Definition \ref{LF}. Take a base point $b_0\in B$ 
and an orientation preserving diffeomorphism $\Phi\co\Sigma_g\rightarrow F_0:=f^{-1}(b_0)$. 
Since $f$ restricted over $B-\Delta$ 
is a smooth fiber bundle with fiber $\Sigma_g$, 
we can define a homomorphism 
\[
\rho\co\pi_1(B-\Delta,b_0)\rightarrow \mathcal{M}_g 
\]
called the {\it monodromy representation} of $f$ with respect to $\Phi$. 
Let $\gamma$ be the loop based at $b_0$ 
consisting of the boundary circle of a small disk neighborhood 
of $b\in\Delta$ oriented counterclockwise and 
a simple path connecting a point on the circle to $b_0$ in $B-\Delta$. 
It is known that $\rho([\gamma])$ is a 
Dehn twist along some essential simple closed curve 
$c$ on $\Sigma_g$, which is called the {\it vanishing cycle} of 
the critical point $p$ on $f^{-1}(b)$. 
If $p$ is of positive type (resp. of negative type), then 
the Dehn twist is right-handed (resp. left-handed). 

A singular fiber is said to be of  {\it type {\rm I}} if the vanishing cycle is non-separating 
and of {\it type ${\rm II}_h$} for $h=1,\ldots , [g/2]$ if the vanishing cycle is 
separating and it bounds a genus--$h$ subsurface of $\Sigma_g$.  
A singular fiber is said to be of {\it type ${\rm I}^+$} (resp. 
{\it type ${\rm I}^-$} and {\it type ${\rm II}_h^+$},  {\it type ${\rm II}_h^-$}) 
if it is of type I and of positive type (resp. of type I and of negative type, 
of type ${\rm II}_h$ and of positive type, of type ${\rm II}_h$ and of negative type). 
We denote by 
$n_0^{+}(f)$, $n_0^{-}(f)$, $n_h^{+}(f)$, and $n_h^{-}(f)$, 
the numbers of singular fibers of $f$ of type 
${\rm I}^+$, ${\rm I}^-$, ${\rm II}_h^+$, and ${\rm II}_h^-$, respectively.  
A Lefschetz fibration is called  {\it irreducible} if 
every singular fiber is of type I.   
A Lefschetz fibration is called {\it chiral} if every singular fiber is of positive type. 

Suppose that the cardinality of $\Delta$ is equal to $n$. 
A system ${\cal A}= (A_1, \dots, A_n)$ of arcs on $B$ is called 
a {\it Hurwitz arc system} for $\Delta$ with base point $b_0$ if 
each $A_i$ is an embedded arc connecting $b_0$ 
with a point of $\Delta$ in $B$ such that $A_i \cap A_j= \{b_0\}$ for $i \ne j$, 
and they appear in this order around $b_0$ (see Kamada \cite{Kamada2002}).  
When $B$ is a $2$-sphere, 
the system ${\cal A}$ determines a system of generators of 
$\pi_1(B- \Delta, b_0)$, say $(a_1, \dots, a_n)$.  
We call $( \rho(a_1), \dots, \rho(a_n) )$ a {\it Hurwitz system} of $f$.

\subsection{Chart description and Wajnryb's presentation}

In this subsection we introduce a chart description for Lefschetz fibrations of genus 
greater than two by employing Wajnryb's finite presentation \cite{Wajnryb1983} 
of mapping class groups. 
General theories of charts for presentations of groups were developed independently by 
Kamada \cite{Kamada2007} and Hasegawa \cite{Hasegawa2006}. 
We use the terminology of chart description in Kamada \cite{Kamada2007}. 

We first review a finite presentation of the mapping class group of a closed oriented 
surface due to Wajnryb. 
For $i=0,1,\ldots ,2g$, let $\zeta_i$ be a right-handed Dehn twist along 
the simple closed curve $c_i$ on $\Sigma_g$ depicted in Figure \ref{curves}. 

\begin{figure}[ht!]
\labellist
\small \hair 2pt
\pinlabel $c_1$ [r] at 0 155
\pinlabel $c_2$ [t] at 45 139
\pinlabel $c_3$ [t] at 70 146
\pinlabel $c_4$ [t] at 98 139
\pinlabel $c_0$ [r] at 88 182
\pinlabel $c_{2g}$ [t] at 242 139
\pinlabel $c_2$ [t] at 44 26
\pinlabel $c_{2h}$ [t] at 105 26
\pinlabel $c_{2h+2}$ [t] at 182 26
\pinlabel $c_{2g}$ [t] at 243 26
\pinlabel $s_{h}$ [r] at 135 72
\endlabellist
\centering
\includegraphics[scale=0.72]{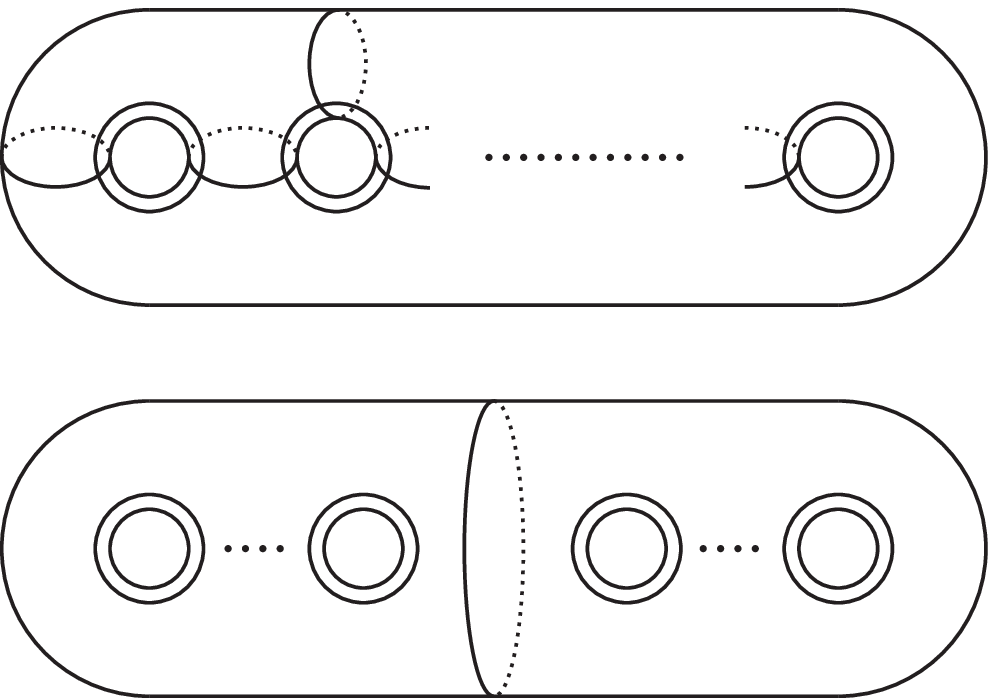}
\caption{Simple closed curves on $\Sigma_g$}
\label{curves}
\end{figure}

\begin{thm}[Wajnryb \cite{Wajnryb1983, Wajnryb1999}]\label{pres} 
Suppose that $g$ is greater than two. 
The mapping class group $\mathcal{M}_g$ is generated by elements 
$\zeta_0,\zeta_1,\zeta_2,\ldots ,\zeta_{2g}$ and has defining relations: 
{\allowdisplaybreaks %
\[\tag*{$\bullet$ (Far--commutation)}
\zeta_i\zeta_j =\zeta_j\zeta_i \;\; (1\leq i<j-1\leq 2g-1), \quad 
\zeta_0\zeta_j  =\zeta_j\zeta_0 \;\; (j=1,2,3,5,\ldots ,2g), 
\]
\[\tag*{$\bullet$ (Braid relation)}
\zeta_i\zeta_{i+1}\zeta_i =\zeta_{i+1}\zeta_i\zeta_{i+1} \;\; (i=1,\ldots ,2g-1), \quad 
\zeta_0\zeta_4\zeta_0 =\zeta_4\zeta_0\zeta_4;
\]
\[\tag*{$\bullet$ ($3$-chain relation)}
(\zeta_3\zeta_2\zeta_1)^4
=\zeta_0
\zeta_4^{-1}\zeta_3^{-1}\zeta_2^{-1}\zeta_1^{-2}\zeta_2^{-1}\zeta_3^{-1}\zeta_4^{-1}
\zeta_0\zeta_4\zeta_3\zeta_2\zeta_1^2\zeta_2\zeta_3\zeta_4;
\]
\[\tag*{$\bullet$ (Lantern relation)}
\delta_3\zeta_1\zeta_3\zeta_5 = \zeta_0\tau_2\zeta_0\tau_2^{-1}
\tau_1\tau_2\zeta_0\tau_2^{-1}\tau_1^{-1}, \quad {\sl where} 
\]
\begin{align*}
& \tau_1:=\zeta_2\zeta_3\zeta_1\zeta_2, \quad 
\tau_2:=\zeta_4\zeta_5\zeta_3\zeta_4, \quad 
\mu:=\zeta_5\zeta_6\tau_2\zeta_0\tau_2^{-1}\zeta_6^{-1}\zeta_5^{-1}, \\
& \nu:=\zeta_1\zeta_2\zeta_3\zeta_4\zeta_0\zeta_4^{-1}\zeta_3^{-1}\zeta_2^{-1}\zeta_1^{-1}, 
\quad \delta_3:=\zeta_6^{-1}\zeta_5^{-1}\zeta_4^{-1}\zeta_3^{-1}\zeta_2^{-1}
\mu^{-1}\nu\mu\zeta_2\zeta_3\zeta_4\zeta_5\zeta_6;
\end{align*}
\[\tag*{$\bullet$ (Hyperelliptic relation)}
\zeta_{2g}\cdots\zeta_3\zeta_2\zeta_1^2\zeta_2\zeta_3\cdots\zeta_{2g}
\delta_g=\delta_g
\zeta_{2g}\cdots\zeta_3\zeta_2\zeta_1^2\zeta_2\zeta_3\cdots\zeta_{2g}, 
\quad {\sl where}
\]
\begin{align*}
& \tau_1:=\zeta_2\zeta_3\zeta_1\zeta_2, \quad 
\tau_i:=\zeta_{2i}\zeta_{2i-1}\zeta_{2i+1}\zeta_{2i}, \\
& \nu_1:=\zeta_4^{-1}\zeta_3^{-1}\zeta_2^{-1}\zeta_1^{-2}\zeta_2^{-1}\zeta_3^{-1}\zeta_4^{-1}
\zeta_0\zeta_4\zeta_3\zeta_2\zeta_1^2\zeta_2\zeta_3\zeta_4, \quad 
\nu_i:=\tau_{i-1}\tau_i\nu_{i-1}\tau_i^{-1}\tau_{i-1}^{-1}, \\
& \mu_1:=\zeta_2\zeta_3\zeta_4\nu_1\zeta_1^{-1}\zeta_2^{-1}\zeta_3^{-1}\zeta_4^{-1}, \quad 
\mu_i:=\zeta_{2i}\zeta_{2i+1}\zeta_{2i+2}\nu_i\zeta_{2i-1}^{-1}\zeta_{2i}^{-1}
\zeta_{2i+1}^{-1}\zeta_{2i+2}^{-1}, \\
& \delta_g :=\mu_{g-1}^{-1}\cdots \mu_2^{-1}\mu_1^{-1}\zeta_1\mu_1\mu_2\cdots\mu_{g-1}
\end{align*}}
for $i=2,\ldots ,g-1$. 
\end{thm}

We make use of the presentation above to introduce a notion of chart 
which gives a graphic description of monodromy representations of Lefschetz fibrations. 
We set 
{\allowdisplaybreaks %
\begin{align*}
\mathcal{X} & :=\{\zeta_0,\zeta_1,\ldots ,\zeta_{2g}\}, \\
\mathcal{R} & :=\{r_F(i,j)\, |\, 1\leq i<j-1\leq 2g-1\}\cup \{r_F(0,j)\, |\, j=1,2,3,5,\ldots ,2g\} \\
& \quad \cup \{r_B(i)\, |\, i=0,1,\ldots ,2g-1\}\cup \{r_C,r_L,r_H \}, \\
\mathcal{S} & :=\{\ell_0(i)^{\pm 1}\, |\, i=0,1,\ldots ,2g\} 
\cup \{\ell_h^{\pm 1}\, |\, h=1,\ldots ,[g/2]\}, 
\end{align*}}
for $g\geq 3$, where 
{\allowdisplaybreaks %
\begin{align*}
& r_F(i,j) :=\zeta_i\zeta_j\zeta_i^{-1}\zeta_j^{-1}, \quad 
r_B(0):=\zeta_0\zeta_4\zeta_0\zeta_4^{-1}\zeta_0^{-1}\zeta_4^{-1}, \\
& r_B(i) :=\zeta_i\zeta_{i+1}\zeta_i\zeta_{i+1}^{-1}\zeta_i^{-1}\zeta_{i+1}^{-1} \;\; (i=1,\ldots ,2g-1), \\
& r_C:=(\zeta_3\zeta_2\zeta_1)^4
\zeta_4^{-1}\zeta_3^{-1}\zeta_2^{-1}\zeta_1^{-2}\zeta_2^{-1}\zeta_3^{-1}\zeta_4^{-1}
\zeta_0^{-1}\zeta_4\zeta_3\zeta_2\zeta_1^2\zeta_2\zeta_3\zeta_4\zeta_0^{-1}, \\
& r_L:=\delta_3\zeta_1\zeta_3\zeta_5\tau_1\tau_2\zeta_0^{-1}\tau_2^{-1}\tau_1^{-1}
\tau_2\zeta_0^{-1}\tau_2^{-1}\zeta_0^{-1}, \\
& r_H:=\zeta_{2g}\cdots\zeta_3\zeta_2\zeta_1^2\zeta_2\zeta_3\cdots\zeta_{2g}
\delta_g\zeta_{2g}^{-1}\cdots\zeta_3^{-1}\zeta_2^{-1}\zeta_1^{-2}
\zeta_2^{-1}\zeta_3^{-1}\cdots\zeta_{2g}^{-1}\delta_g^{-1}, \\
& \ell_0(i):=\zeta_i \;\; (i=0,1,\ldots ,2g), \quad 
\ell_h:=(\zeta_1\zeta_2\cdots\zeta_{2h})^{4h+2} \;\; (h=1,\ldots ,[g/2]), 
\end{align*}}
and $\delta_3, \tau_1, \tau_2, \delta_g$ are defined as in Theorem \ref{pres}. 

Let $B$ be a connected closed oriented surface 
and $\Gamma$ a finite graph in $B$ such that each edge of $\Gamma$ is oriented 
and labeled with an element of $\mathcal{X}$. 
We denote the label $\zeta_i$ by $i$ for short. 
Choose a simple path $\gamma$ which intersects with edges of $\Gamma$ 
transversely and does not intersect with vertices of $\Gamma$. 
For such a path $\gamma$, we obtain a word 
$w_{\Gamma}(\gamma)$ in $\mathcal{X}\cup\mathcal{X}^{-1}$ 
by reading off the labels of intersecting edges along $\gamma$ with exponents 
as in Figure \ref{intersection} (a). 
We call the word $w_{\Gamma}(\gamma)$ the {\it intersection word} of $\gamma$ 
with respect to $\Gamma$. 
Conversely, we can specify the number, orientations, and labels of consecutive edges in $\Gamma$ 
by indicating a (dashed) arrow intersecting the edges transversely 
together with the intersection 
word of the arrow with respect to $\Gamma$ (see Figure \ref{intersection} (b) and (c)). 

\begin{figure}[ht!]
\labellist
\small \hair 2pt
\pinlabel $\gamma$ [r] at 0 18
\pinlabel $1$ [t] at 30 0
\pinlabel $2$ [t] at 59 0
\pinlabel $1$ [t] at 87 0
\pinlabel $3$ [t] at 116 0
\pinlabel $2$ [t] at 143 0
\pinlabel {(a)} [t] at 85 -15
\pinlabel $w$ [l] at 255 18
\pinlabel {(b)} [t] at 235 -15
\pinlabel $w$ [b] at 313 33
\pinlabel {(c)} [t] at 313 -15
\endlabellist
\centering
\includegraphics[scale=0.9]{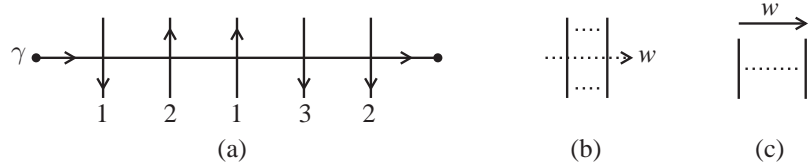}
\vspace{5mm}
\caption{Intersection word 
$w_{\Gamma}(\gamma)=w=\zeta_1\zeta_2^{-1}\zeta_1^{-1}\zeta_3\zeta_2$}
\label{intersection}
\end{figure}

For a vertex $v$ of $\Gamma$, a small simple closed curve surrounding $v$ in the 
counterclockwise direction is called a {\it meridian loop} of $v$ and denoted by $m_v$. 
The vertex $v$ is said to be {\it marked} if one of the regions around 
$v$ is specified by an asterisk. 
If $v$ is marked, the intersection word $w_{\Gamma}(m_v)$ of $m_v$ 
with respect to $\Gamma$ is well-defined. 
If not, it is determined up to cyclic permutation. See Kamada \cite{Kamada2007} 
for details. 

\begin{defn}\label{def:chart} 
A {\it chart} in $B$ is a finite graph $\Gamma$ in $B$ 
(possibly being empty or having {\it hoops} that are closed edges without vertices) 
whose edges are labeled with an element of $\mathcal{X}$, 
and oriented so that the following conditions are satisfied 
(see Figure \ref{verticesA}, Figure \ref{verticesB}, and Figure \ref{verticesC}):   
\begin{itemize}
\item[(1)] the vertices of $\Gamma$ are classified into two families: 
{\it white vertices} and {\it black vertices};  
\item[(2)] if $v$ is a white vertex (resp. a black vertex), 
the word $w_{\Gamma}(m_v)$ is a cyclic permutation of an element of 
$\mathcal{R}\cup\mathcal{R}^{-1}$ (resp. of $\mathcal{S}$). 
\end{itemize}
A white vertex $v$ is said to be of 
{\it type} $r$ (resp. of {\it type} $r^{-1}$) if 
$w_{\Gamma}(m_v)^{-1}$ is a cyclic permutation of $r\in\mathcal{R}$ 
(resp. of $r^{-1}\in\mathcal{R}^{-1}$). 
A black vertex $v$ is said to be of {\it type} $s$ if 
$w_{\Gamma}(m_v)$ is a cyclic permutation of $s\in\mathcal{S}$. 
A chart $\Gamma$ is said to be {\it marked} if each white vertex (resp. black vertex) $v$ is 
marked and $w_{\Gamma}(m_v)$ is exactly an element of 
$\mathcal{R}\cup\mathcal{R}^{-1}$ (resp. of $\mathcal{S}$). 
If a base point $b_0$ of $B$ is specified, we always assume that a chart $\Gamma$ is 
disjoint from $b_0$. 
A chart consisting of two black vertices and one edge connecting them is called 
a {\it free edge}. 
\end{defn}

\begin{figure}[ht!]
\labellist
\footnotesize \hair 2pt
\pinlabel $i$ [tr] at 1 21
\pinlabel $j$ [br] at 1 52
\pinlabel $i$ [bl] at 31 52
\pinlabel $j$ [tl] at 31 21
\pinlabel $i$ [tr] at 95 18
\pinlabel $i\negthinspace +\negthinspace 1$ [r] at 88 35
\pinlabel $i$ [br] at 96 54
\pinlabel $i\negthinspace +\negthinspace 1$ [bl] at 133 54
\pinlabel $i$ [l] at 139 35
\pinlabel $i\negthinspace +\negthinspace 1$ [tl] at 132 18
\pinlabel $3$ [b] at 222 73
\pinlabel $2$ [b] at 231 73
\pinlabel $1$ [b] at 239 73
\pinlabel $3$ [b] at 248 73
\pinlabel $2$ [b] at 256 73
\pinlabel $1$ [b] at 265 73
\pinlabel $3$ [b] at 273 73
\pinlabel $2$ [b] at 282 73
\pinlabel $1$ [b] at 291 73
\pinlabel $3$ [b] at 300 73
\pinlabel $2$ [b] at 309 73
\pinlabel $1$ [b] at 317 73
\pinlabel $4$ [t] at 346 0
\pinlabel $3$ [t] at 337 0
\pinlabel $2$ [t] at 330 0
\pinlabel $1$ [t] at 321 0
\pinlabel $1$ [t] at 312 0
\pinlabel $2$ [t] at 303 0
\pinlabel $3$ [t] at 295 0
\pinlabel $4$ [t] at 287 0
\pinlabel $0$ [t] at 278 0
\pinlabel $4$ [t] at 270 0
\pinlabel $3$ [t] at 261 0
\pinlabel $2$ [t] at 252 0
\pinlabel $1$ [t] at 243 0
\pinlabel $1$ [t] at 235 0
\pinlabel $2$ [t] at 227 0
\pinlabel $3$ [t] at 219 0
\pinlabel $4$ [t] at 210 0
\pinlabel $0$ [t] at 201 0
\endlabellist
\centering
\includegraphics[scale=0.8]{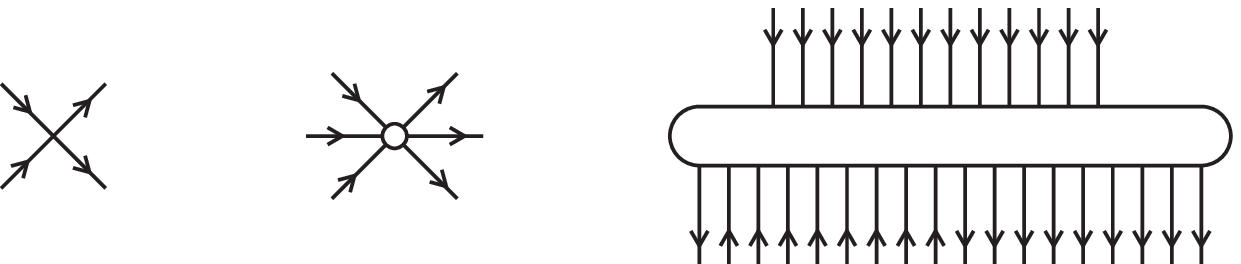}
\caption{Vertices of type $r_F(i,j)$, $r_B(i)\; (i\ne 0)$, $r_C$}
\label{verticesA}
\end{figure}

\begin{figure}[ht!]
\labellist
\footnotesize \hair 2pt
\pinlabel $\delta_3$ [b] at 29 90
\pinlabel $1$ [b] at 47 84
\pinlabel $3$ [b] at 56 84
\pinlabel $5$ [b] at 64 84
\pinlabel $\tau_1$ [b] at 81 90
\pinlabel $\tau_2$ [b] at 107 90
\pinlabel $0$ [t] at 128 9
\pinlabel $\tau_2^{-\negthinspace 1}$ [t] at 111 4
\pinlabel $\tau_1^{-\negthinspace 1}$ [t] at 85 4
\pinlabel $\tau_2$ [t] at 62 4
\pinlabel $0$ [t] at 42 9
\pinlabel $\tau_2^{-\negthinspace 1}$ [t] at 26 4
\pinlabel $0$ [t] at 8 9
\pinlabel $2g$ [b] at 218 94
\pinlabel $1$ [b] at 244 94
\pinlabel $1$ [b] at 253 94
\pinlabel $2g$ [b] at 278 94
\pinlabel $\delta_g$ [l] at 321 47
\pinlabel $2g$ [t] at 278 0
\pinlabel $1$ [t] at 253 0
\pinlabel $1$ [t] at 244 0
\pinlabel $2g$ [t] at 218 0
\pinlabel $\delta_g^{-1}$ [r] at 178 47
\endlabellist
\centering
\includegraphics[scale=0.8]{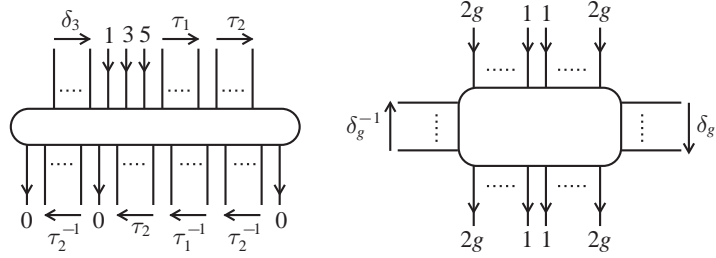}
\caption{Vertices of type $r_L$ and $r_H$}
\label{verticesB}
\end{figure}

\begin{figure}[ht!]
\labellist
\footnotesize \hair 2pt
\pinlabel $i$ [l] at 46 34
\pinlabel $i$ [l] at 46 8
\pinlabel $1$ [b] at 223 45
\pinlabel $2h$ [b] at 202 45
\pinlabel $1$ [b] at 193 45
\pinlabel $2h$ [b] at 174 45
\pinlabel $1$ [b] at 139 45
\pinlabel $2h$ [b] at 118 45
\endlabellist
\centering
\includegraphics[scale=0.8]{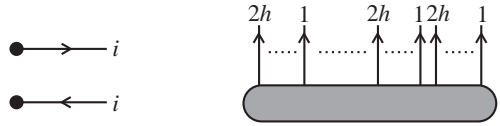}
\caption{Vertices of type $\ell_0(i)^{\pm 1}$ and $\ell_h$}
\label{verticesC}
\end{figure}

\begin{rem} 
It would be worth noting that the intersection word of a `clockwise' meridian of 
a white vertex of type $r$ is equal to $r$, while that of a `counterclockwise' meridian 
of a black vertex of type $s$ is equal to $s$ in this paper. 
This notation is different from those of Kamada \cite{Kamada2007} 
and Hasegawa \cite{Hasegawa2006}, 
who always consider `counterclockwise' meridians for both white and black vertices. 
\end{rem}

We next introduce several moves for charts. 
Let $\Gamma$ and $\Gamma'$ be two charts on $B$ and $b_0$ a base point of $B$. 

Let $D$ be a disk embedded in $B-\{b_0\}$. 
Suppose that the boundary $\partial D$ of $D$ intersects $\Gamma$ and $\Gamma'$ 
transversely. 

\begin{defn} 
We say that $\Gamma'$ is obtained from $\Gamma$ by a {\it chart move of type W} 
if $\Gamma\cap (B-{\rm Int}\, D)=\Gamma'\cap (B-{\rm Int}\, D)$ and 
that both $\Gamma\cap D$ and $\Gamma'\cap D$ have no black vertices. 
We call chart moves of type W shown in Figure \ref{movesA} (a), (b), and (c), 
a {\it channel change}, a {\it birth/death of a hoop}, and 
a {\it birth/death of a pair of white vertices}, respectively. 
\end{defn}

\begin{figure}[ht!]
\labellist
\footnotesize \hair 2pt
\pinlabel $i$ [tr] at 0 66
\pinlabel $i$ [bl] at 41 106
\pinlabel $i$ [tr] at 80 66
\pinlabel $i$ [bl] at 121 106
\pinlabel (a) [t] at 61 84
\pinlabel $i$ [tl] at 184 76
\pinlabel empty at 258 86
\pinlabel (b) [t] at 209 84
\pinlabel $r$ at 57 20
\pinlabel $r^{\negthinspace -\negthinspace 1}$ at 87 23
\pinlabel (c) [t] at 132 19
\endlabellist
\centering
\includegraphics[scale=0.8]{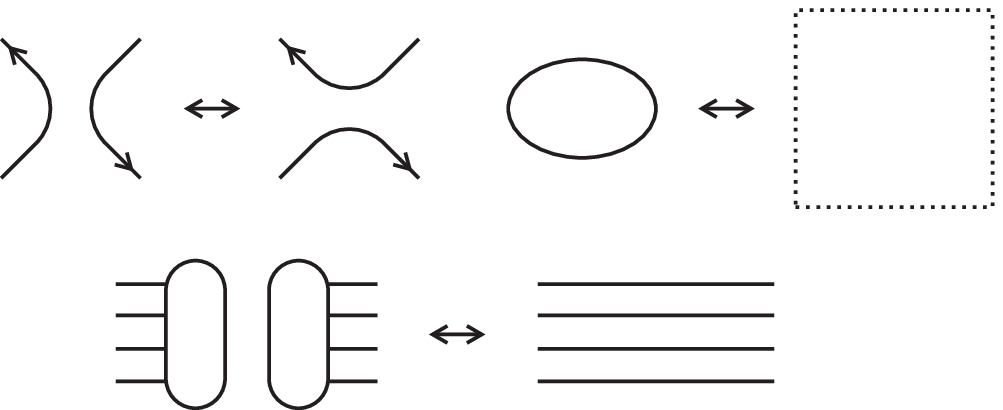}
\caption{Chart moves of type W}
\label{movesA}
\end{figure}

Let $s$ and $s'$ be elements of $\mathcal{S}$. 
Suppose that there exists a word $w$ in $\mathcal{X}\cup\mathcal{X}^{-1}$ such that 
two words $s'$ and $wsw^{-1}$ determine the same element of $\mathcal{M}_g$. 

\begin{defn} 
If a chart $\Gamma$ contains a black vertex of type $s$, 
then we can change a part of $\Gamma$ near the vertex by using a local replacement 
depicted in Figure \ref{movesB} to obtain another chart $\Gamma'$. 
We say that $\Gamma'$ is obtained from $\Gamma$ by a {\it chart move of transition}. 
Note that the blank labeled with ${\rm T}$ can be filled only with edges and white vertices. 
\end{defn}

\begin{figure}[ht!]
\labellist
\footnotesize \hair 2pt
\pinlabel $s$ [b] at 35 81
\pinlabel $s'$ [b] at 205 81
\pinlabel $w$ [bl] at 245 86
\pinlabel $w$ [tl] at 245 40
\pinlabel $s$ [bl] at 257 81
\pinlabel T at 232 64
\endlabellist
\centering
\includegraphics[scale=0.8]{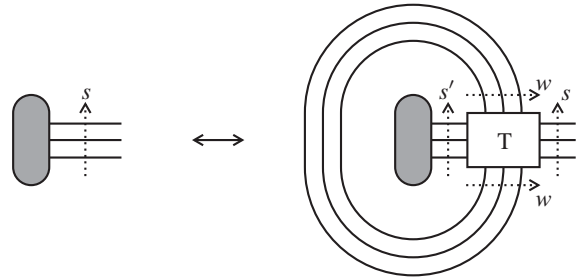}
\caption{Chart move of transition}
\label{movesB}
\end{figure}

\begin{defn} 
We say that $\Gamma'$ is obtained from $\Gamma$ by a {\it chart move of conjugacy type} 
if $\Gamma'$ is obtained from $\Gamma$ 
by a local replacement depicted in Figure \ref{movesC}. 
\end{defn}

\begin{figure}[ht!]
\labellist
\footnotesize \hair 2pt
\pinlabel $b_0$ [l] at 7 22
\pinlabel $b_0$ [l] at 117 22
\pinlabel $i$ [l] at 135 22
\pinlabel $b_0$ [l] at 197 22
\pinlabel $b_0$ [l] at 304 22
\pinlabel $i$ [l] at 323 22
\endlabellist
\centering
\includegraphics[scale=0.8]{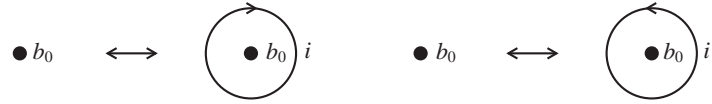}
\caption{Chart moves of conjugacy type}
\label{movesC}
\end{figure}

Let $\Gamma$ be a chart in $B$ with base point $b_0$ 
and $\Delta_{\Gamma}$ the set of black vertices of $\Gamma$. 
For a loop $\gamma$ in $B-\Delta_{\Gamma}$ based at $b_0$, 
the element of $\mathcal{M}_g$ determined by 
the intersection word $w_{\Gamma}(\gamma)$ of $\gamma$ with respect to $\Gamma$ does not 
depend on a choice of representative of the homotopy class of $\gamma$. 
Thus we obtain a homomorphism 
$\rho_{\Gamma}\co\pi_1(B-\Delta_{\Gamma},b_0)\rightarrow \mathcal{M}_g$, 
which is called the {\it homomorphism determined by $\Gamma$}. 

We now state a classification of Lefschetz fibrations in terms of charts and chart moves. 
Let $B$ be a connected closed oriented surface. 

\begin{prop}\label{correspondence} 
Suppose that $g$ is greater than two. 
{\rm (1)} Let $f$ be a Lefschetz fibration of genus $g$ over $B$ and 
$\rho$ a monodromy representation of $f$. 
Then there exists a chart $\Gamma$ in $B$ such that the homomorphism $\rho_{\Gamma}$ 
determined by $\Gamma$ is equal to $\rho$. 
{\rm (2)} For every chart $\Gamma$ in $B$, there exists a Lefschetz fibration 
$f$ of genus $g$ over $B$ such that a monodromy representation of $f$ 
is equal to the homomorphism $\rho_{\Gamma}$ determined by $\Gamma$. 
\end{prop}

We call such $\Gamma$ as in Proposition \ref{correspondence} (1) 
a chart {\it corresponding to $f$}, 
and such $f$ as in Proposition \ref{correspondence} (2) 
a Lefschetz fibration {\it described by $\Gamma$}. 

Instead of giving a proof of Proposition \ref{correspondence}, 
we show an example of a chart and describe the correspondence of the chart to 
a Hurwitz system of a Lefschetz fibration. 

\begin{exam} Let $B$ be a $2$--sphere. 
We consider a chart $\Gamma$ in $B$ with base point $b_0$ and 
a system $(\gamma_1,\gamma_2,\gamma_3,\gamma_4)$ of loops based at $b_0$, 
which is determined by a Hurwitz arc system $\mathcal{A}$ 
for the set $\Delta_{\Gamma}$ of black vertices of $\Gamma$, 
as in Figure \ref{exampleA}. 
The intersection words of the loops with respect to $\Gamma$ are 
\begin{align*}
w_{\Gamma}(\gamma_1) & =\zeta_1^{-1}\zeta_2^{-1}\zeta_1\zeta_2\zeta_1, \quad 
w_{\Gamma}(\gamma_2) =\zeta_1^{-1}\zeta_3\zeta_1, \\
w_{\Gamma}(\gamma_3) & =\zeta_2^{-1}\zeta_3^{-1}\zeta_2^{-1}\zeta_3\zeta_2, \quad 
w_{\Gamma}(\gamma_4) =\zeta_2^{-1}, 
\end{align*}
each of which represents the image $\rho_{\Gamma}(a_i)$ of the homotopy class 
$a_i$ of $\gamma_i$ under  the homomorphism 
$\rho_{\Gamma}:\pi_1(B-\Delta_{\Gamma},b_0)\rightarrow \mathcal{M}_g$. 
Since the group $\pi_1(B-\Delta_{\Gamma},b_0)$ has a presentation 
$\langle a_1,a_2,a_3,a_4\, |\, a_1a_2a_3a_4=1 \rangle$, 
$\rho_{\Gamma}$ is determined by the system 
$(\rho_{\Gamma}(a_1),\rho_{\Gamma}(a_2),\rho_{\Gamma}(a_3), \rho_{\Gamma}(a_4))$, 
which is a Hurwitz system of a certain Lefschetz fibration of genus $g$ over $B$ 
because each $\rho_{\Gamma}(a_i)$ is a Dehn twist. 
Note that the product $w_{\Gamma}(\gamma_1)w_{\Gamma}(\gamma_2)
w_{\Gamma}(\gamma_3)w_{\Gamma}(\gamma_4)$ of the intersection words 
represents the identity of $\mathcal{M}_g$. 
\end{exam}

\begin{figure}[ht!]
\labellist
\footnotesize \hair 2pt
\pinlabel $b_0$ [t] at 175 184
\pinlabel $1$ [b] at 30 156
\pinlabel $2$ [b] at 46 95
\pinlabel $1$ [b] at 112 77
\pinlabel $2$ [b] at 186 77
\pinlabel $1$ [r] at 154 103
\pinlabel $3$ [t] at 137 129
\pinlabel $3$ [b] at 191 115
\pinlabel $3$ [r] at 208 25
\pinlabel $2$ [r] at 243 50
\pinlabel $2$ [t] at 283 114
\pinlabel {\small $\gamma_1$} [r] at 70 135
\pinlabel {\small $\gamma_2$} [t] at 152 166
\pinlabel {\small $\gamma_3$} [t] at 233 142
\pinlabel {\small $\gamma_4$} [b] at 278 162
\endlabellist
\centering
\includegraphics[scale=0.75]{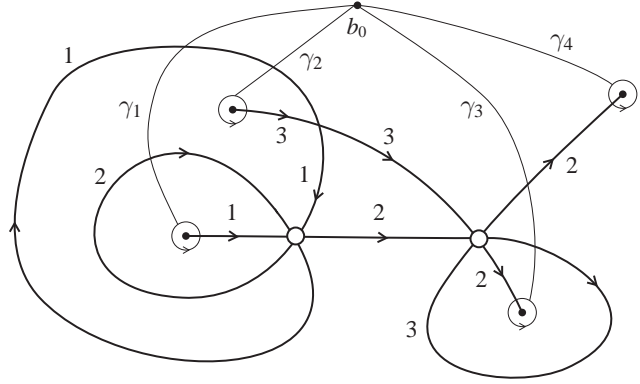}
\caption{Monodromy of a chart $\Gamma$}
\label{exampleA}
\end{figure}

\begin{thm}\label{classification} 
Suppose that $g$ is greater than two. 
Let $f$ and $f'$ be Lefschetz fibrations of genus $g$ over $B$, 
and $\Gamma$ and $\Gamma'$ 
charts corresponding to $f$ and $f'$, respectively. 
Then $f$ is strictly isomorphic to $f'$ if and only if 
$\Gamma$ is transformed to $\Gamma'$ by a finite sequence of chart moves of type W, 
chart moves of transitions, chart moves of conjugacy type, 
and ambient isotopies of $B$ relative to $b_0$. 
\end{thm}

Proposition \ref{correspondence} and Theorem \ref{classification} 
follow from a classification theorem of Lefschetz fibrations due to 
Kas \cite{Kas1980} and Matsumoto \cite{Matsumoto1996} 
together with fundamental theorems on charts and chart moves by Kamada 
\cite[Sections 4--8]{Kamada2007}. 

We end this subsection with a definition and chart description of fiber sums of 
Lefschetz fibrations. 
Let $f\co M\rightarrow B$ and $f'\co M'\rightarrow B'$ be Lefschetz fibrations of genus $g$. 
Take regular values $b_0\in B$ and $b_0'\in B'$ of $f$ and $f'$, 
and small disks $D_0\subset B-\Delta$ and $D'_0\subset B-\Delta'$ 
near $b_0$ and $b'_0$, respectively. 
Consider general fibers $F_0:=f^{-1}(b_0)$ and $F'_0:=f'^{-1}(b_0')$ 
and orientation preserving diffeomorphisms $\Phi\co \Sigma_g\rightarrow F_0$ 
and $\Phi'\co\Sigma_g\rightarrow F'_0$, respectively. 

\begin{defn} 
Let $\Psi\co \Sigma_g\rightarrow \Sigma_g$ be an orientation preserving diffeomorphism 
and $r\co \partial D_0\rightarrow \partial D'_0$ an orientation reversing diffeomorphism. 
The new manifold $M\#_F M'$ obtained by glueing 
$M-f^{-1}({\rm Int}\, D_0)$ and $M'-f'^{-1}({\rm Int}\, D'_0)$ by 
$(\Phi'\circ\Psi\circ\Phi^{-1})\times r$ 
admits a Lefschetz fibration 
$f\#_{\Psi}\, f'\co M\#_F M'\rightarrow B\# B'$ of genus $g$. 
We call $f\#_{\Psi}\, f'$ the {\it fiber sum} of $f$ and $f'$ with respect to $\Psi$. 
Although 
the diffeomorphim type of $M\#_F M'$ and the isomorphism type of $f\#_{\Psi}\, f'$ 
depend on a choice of the diffeomorphism $\Psi$ in general, 
we often abbreviate $f\#_{\Psi}\, f'$ as $f\#\, f'$. 
\end{defn}

Let $\Gamma$ and $\Gamma'$ be charts corresponding to $f$ and $f'$, 
and $D_0$ and $D'_0$ small disks near $b_0$ and $b_0'$ disjoint from 
$\Gamma$ and $\Gamma'$, respectively. 
Connecting $B-{\rm Int}\, D_0$ with $B'-{\rm Int}\, D'_0$ by a tube, 
we have a connected sum $B\# B'$ of $B$ and $B'$. 
Let $w$ be a word in $\mathcal{X}\cup\mathcal{X}^{-1}$ which represents 
the mapping class of $\Psi$ in $\mathcal{M}_g$. 
Let $\Gamma\#_w\Gamma'$ be the union of $\Gamma$, $\Gamma'$, and 
hoops on the tube representing $w$ (see Figure \ref{fibersum}). 
Then the fiber sum $f\#_{\Psi}\, f'$ is described by 
this new chart $\Gamma\#_w\Gamma'$ in $B\# B'$ with base point $b_0$. 
If the word $w$ is trivial, then the chart $\Gamma\#_w\Gamma'$ is denoted also 
by $\Gamma\oplus\Gamma'$, 
which is called a {\it product} of $\Gamma$ and $\Gamma'$. 

\begin{figure}[ht!]
\labellist
\footnotesize \hair 2pt
\pinlabel $b_0$ [r] at 128 38
\pinlabel $b'_0$ [l] at 210 38
\pinlabel $B$ [br] at 3 57
\pinlabel $B'$ [bl] at 330 59
\pinlabel $w$ [t] at 170 13
\pinlabel $\Gamma$ at 71 40
\pinlabel $\Gamma'$ at 270 40
\endlabellist
\centering
\includegraphics[scale=0.8]{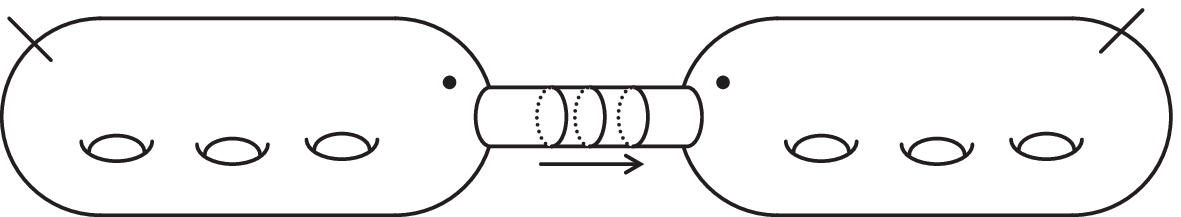}
\caption{Chart $\Gamma\#_w\Gamma'$ in $B\# B'$}
\label{fibersum}
\end{figure}


\section{Signature of Lefschetz fibrations}


In this section we review the signature cocycle discovered by Meyer 
and prove a signature theorem for Lefschetz fibrations. 

\subsection{Meyer's signature cocycle}

In this subsection we give a brief survey on Meyer's signature cocycle. 
We begin with the definition of the signature cocycle. 
Let $g$ be a positive integer. 

\begin{defn}[Meyer \cite{Meyer1973}]\label{cocycle} 
For $A,B\in {\rm Sp}(2g,\Bbb{Z})$, we consider the vector space
\[
V_{A,B}:=\{ (x,y)\in \Bbb{R}^{2g}\times\Bbb{R}^{2g}\, |\, (A^{-1}-I_{2g})x+(B-I_{2g})y=0\}
\]
and the bilinear form $\langle \; ,\; \rangle_{A,B}\co V_{A,B}\times V_{A,B}\rightarrow\Bbb{R}$ 
defined by 
\[
\langle (x_1,y_1),(x_2,y_2)\rangle_{A,B}:=(x_1+y_1)\cdot J(I_{2g}-B)y_2, 
\]
where $\cdot$ is the standard inner product of $\Bbb{R}^{2g}$ and 
$J=\left(\begin{smallmatrix}
0 & I_g \\
-I_g & 0 
\end{smallmatrix}\right)$. Since $\langle \; ,\; \rangle_{A,B}$ is symmetric, we can define 
an integer $\tau_g(A,B)$ to be the signature of $(V_{A,B}, \langle \; ,\; \rangle_{A,B})$. 
The map $\tau_g\co {\rm Sp}(2g,\Bbb{Z})\times {\rm Sp}(2g,\Bbb{Z})\rightarrow \Bbb{Z}$ 
is called the {\it signature cocycle}. 
\end{defn}

Let $P$ be a compact connected oriented surface of genus $0$ with three boundary 
components and $\pi\co E\rightarrow P$ 
a fiber bundle over $P$ with fiber $\Sigma_g$ and structure group ${\rm Diff}_+\Sigma_g$. 
The fundamental group $\pi_1(P,*)$ of $P$ with base point $*$ is a free group generated 
by two loops $a$ and $b$ depicted in Figure \ref{pants}. 
If we take an orientation preserving diffeomorphism $\Sigma_g\rightarrow \pi^{-1}(*)$, 
we obtain the monodromy representation $\pi_1(P,*)\rightarrow \mathcal{M}_g$ which sends 
$a$ to $\alpha$ and $b$ to $\beta$. 
Since $\mathcal{M}_g^*$ acts on $H:=H_1(\Sigma_g;\Bbb{Z})$ preserving the intersection form, 
we have a representation $\mathcal{M}_g^*\rightarrow {\rm Sp}(2g,\Bbb{Z})$ 
by fixing a symplectic basis on $H$. Let $A$ and $B$ denote matrices corresponding to 
$\alpha$ and $\beta$, respectively. 

\begin{figure}[ht!]
\labellist
\footnotesize \hair 2pt
\pinlabel $P_1$ [br] at 1 56
\pinlabel $P_2$ [bl] at 117 53
\pinlabel $*$ [b] at 58 59
\pinlabel $a$ [tl] at 41 16
\pinlabel $b$ [tr] at 75 16
\endlabellist
\centering
\includegraphics[scale=0.8]{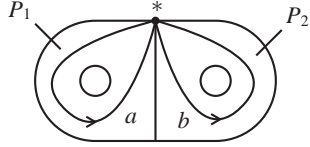}
\caption{Pair of pants $P$}
\label{pants}
\end{figure}

Meyer closely studied the signature of the total space $E$ to obtain the following theorem. 

\begin{thm}[Meyer \cite{Meyer1973}]\label{meyer} 
The signature $\sigma(E)$ of $E$ is equal to $-\tau_g(A,B)$. 
\end{thm}

Theorem \ref{meyer} and Novikov's additivity implies that 
$\tau_g$ is a $2$--cocycle of ${\rm Sp}(2g,\Bbb{Z})$. 

We recall a Maslov index for a triple of Lagrangian subspaces and 
Wall's non-additivity theorem, which are used in the proof of Theorem \ref{meyer}. 

Let $V$ be a real vector space of dimension $2n$, 
$\omega\in \Lambda^2V^*$ a symplectic form on $V$, and 
$\Lambda(V,\omega)$ the Lagrangian Grassmannian of $(V,\omega)$, 
which is the set of Lagrangian subspaces of $(V,\omega)$. 
For $L_1,L_2,L_3\in \Lambda(V,\omega)$, the bilinear form
\begin{align*}
\Psi & \co (L_3+L_1)\cap L_2\times (L_3+L_1)\cap L_2\rightarrow \Bbb{R}\co 
(v,w)\mapsto \omega(v,w_3) \\
& (v,w\in (L_3+L_1)\cap L_2, \, w=w_1+w_3 \; (w_1\in L_1,w_3\in L_3)) 
\end{align*}
is symmetric. We define an integer $i(L_1,L_2,L_3)$ to be the signature of 
$((L_3+L_1)\cap L_2, \Psi)$, which 
is called the {\it ternary Maslov index} of the triple $(L_1,L_2,L_3)$. 

Let $M_1, M_2$ be compact oriented smooth $4$--manifolds, 
$X_1,X_2,X_3$ compact oriented smooth $3$--manifolds, 
and $\Sigma$ a closed oriented smooth $2$--manifold. 
We assume that $M=M_1\cup M_2, \, \partial M_1=X_1\cup X_2,\, \partial M_2=X_2\cup X_3,\, 
\partial X_1=\partial X_2=\partial X_3=\Sigma$, and 
the orientations of these manifolds satisfy 
\begin{gather*}
[M]=[M_1]+[M_2], \; 
\partial_*[M_1]=[X_2]-[X_1], \; \partial_*[M_2]=[X_3]-[X_2], \\
\partial_*[X_1] =\partial_*[X_2]=\partial_*[X_3]=[\Sigma].
\end{gather*}
Let $\omega\co V\times V\rightarrow \Bbb{R}$ 
be the intersection form on $V:=H_1(\Sigma;\Bbb{R})$ 
and $L_i$ the kernel of the homomorphism $V\rightarrow H_1(X_i;\Bbb{R})$ 
induced by the inclusion $\Sigma\rightarrow X_i$ for $i=1,2,3$. 
Since $L_i\in\Lambda(V,\omega)$ for $i=1,2,3$, we can define the ternary 
Maslov index $i(L_1,L_2,L_3)$ of the triple $(L_1,L_2,L_3)$.

\begin{thm}[Wall \cite{Wall1969}]\label{wall}
$\sigma(M)=\sigma(M_1)+\sigma(M_2)-i(L_1,L_2,L_3)$. 
\end{thm}

Gambaudo and Ghys \cite{GG2005} 
(and independently the first author) made use of Theorem \ref{wall} 
to give the following proof of Theorem \ref{meyer}. 
See also Gilmer and Masbaum \cite{GM2011}. 

\begin{proof}[Proof of Theorem \ref{meyer}] 
Consider $P$ to be a boundary sum of two annuli $P_1$ and $P_2$ (see Figure \ref{pants}). 
We set $M:=E$, $M_i:=\pi^{-1}(P_i)\, (i=1,2)$, $X_2:=M_1\cap M_2$, 
$X_1:=\partial M_1-{\rm Int}\, X_2$, $X_3:=\partial M_3-{\rm Int}\, X_2$, 
and $\Sigma:=\partial X_2$. 
Applying Theorem \ref{wall} to these manifolds, 
we have 
\[
\sigma(E)=\sigma(M_1)+\sigma(M_2)-i(L_1,L_2,L_3)
=-i(L_1,L_2,L_3)
\]
because each of $M_1$ and $M_2$ is a product of a mapping torus with an interval, 
which has signature zero. 
Since the bordered component of $X_i$ is diffeomorphic to $I\times\Sigma_g$ 
for $i=1,2,3$, we put $V:=H\oplus H$, $\omega:=\mu\oplus (-\mu)$, and obtain 
\begin{align*}
L_1 & =\{ (-\xi, \alpha_*^{-1}(\xi))\in V\, |\, \xi\in H\}, \quad 
L_2=\{ (-\xi,\xi)\in V\, |\, \xi\in H\}, \\
L_3 & =\{ (-\xi, \beta_*(\xi))\in V\, |\, \xi\in H\}, 
\end{align*}
where $H$ is the first homology $H_1(\Sigma_g;\Bbb{R})$ of $\Sigma_g$ 
and $\mu\co H\times H\rightarrow\Bbb{R}$ 
is the intersection form of $\Sigma_g$. It is easily seen that 
the subspace $(L_1+L_3)\cap L_2$ is written as
\[
(L_1+L_3)\cap L_2 =\{ (-\xi-\eta, \alpha_*^{-1}(\xi)+\beta_*(\eta))\in V \, |\, 
\xi+\eta=\alpha_*^{-1}(\xi)+\beta_*(\eta)\, (\xi,\eta\in H) \}
\]
and the symmetric bilinear form $\Psi$ on $(L_1+L_3)\cap L_2$ is written as
\[
\Psi((-\xi-\eta,\alpha_*^{-1}(\xi)+\beta_*(\eta)),(-\xi'-\eta',\alpha_*^{-1}(\xi')+\beta_*(\eta'))) 
=
\mu(\xi+\eta,({\rm id}-\beta_*)(\eta')). 
\]
We consider the vector space 
\[
U_{\alpha,\beta}:=\{ (\xi,\eta)\in V\, |\, 
(\alpha_*^{-1}-{\rm id})(\xi)+(\beta_*-{\rm id})(\eta)=0 \}
\]
and the symmetric bilinear form $\langle \; ,\;\rangle_{\alpha,\beta}$ 
on $U_{\alpha,\beta}$ defined by 
\[
\langle (\xi,\eta),(\xi',\eta')\rangle_{\alpha,\beta}
:=\mu(\xi+\eta,({\rm id}-\beta_*)(\eta')) 
\quad 
((\xi,\eta),(\xi',\eta')\in U_{\alpha,\beta}). 
\] 
Since the linear map $U_{\alpha,\beta}\rightarrow (L_1+L_3)\cap L_2
\co (\xi,\eta)\mapsto (-\xi-\eta,\xi+\eta)$ is compatible with the bilinear forms, 
the signature of $((L_1+L_3)\cap L_2,\Psi)$ is equal to that of 
$(U_{\alpha,\beta},\langle \; ,\; \rangle_{\alpha,\beta})$, 
which is isomorphic to $(V_{A,B},\langle \; ,\; \rangle_{A,B})$ 
under a choice of a symplectic basis of $H$. 
Therefore we conclude that 
$i(L_1,L_2,L_3)=\tau_g(A,B)$. 
\end{proof}

\begin{rem} It is known that $\tau_g$ is a normalized, symmetric $2$--cocycle of 
${\rm Sp}(2g,\Bbb{Z})$ and invariant under conjugation. 
The cohomology class $[\tau_g]\in H^2({\rm Sp}(2g,\Bbb{Z});\Bbb{Z})$ 
corresponds to $-4c_1$ under homomorphisms: 
\[
H^2({\rm Sp}(2g,\Bbb{Z});\Bbb{Z})\leftarrow 
H^2(B{\rm Sp}(2g,\Bbb{R});\Bbb{Z})\cong 
H^2(BU(g);\Bbb{Z})\cong \Bbb{Z}.
\]
For more details see Meyer \cite{Meyer1973}, Turaev \cite{Turaev1987}, 
Barge and Ghys \cite{BG1992}, and Kuno \cite{Kuno2012}. 
\end{rem}

\subsection{A signature formula}

In this subsection we describe the signature of a Lefschetz fibration of genus 
greater than two in terms of charts. 
Let $g$ be an integer greater than two. 

Let $B$ be a connected closed oriented surface and $\Gamma$ a chart in $B$. 
We denote the number of white vertices of type $r_F(i,j)$ 
(resp. $r_B(i),\, r_C,\, r_L,\, r_H$) 
minus the number of white vertices 
of type $r_F(i,j)^{-1}$ 
(resp. $r_B(i)^{-1},\, r_C^{-1},\, r_L^{-1},\, r_H^{-1}$)
included in $\Gamma$ by $n_F(i,j)(\Gamma)$ 
(resp. $n_B(i)(\Gamma),\, n_C(\Gamma),\, n_L(\Gamma),\, n_H(\Gamma)$). 
Similarly, 
we denote the number of black vertices of type $\ell_0(i)^{\pm 1}$ 
(resp. $\ell_h^{\pm 1}$) included in $\Gamma$ 
by $n_0^{\pm}(i)(\Gamma)$ (resp. $n_h^{\pm}(\Gamma)$), 
and set $n_0(i)(\Gamma):=n_0^+(i)(\Gamma)-n_0^-(i)(\Gamma)$ 
(resp. $n_h(\Gamma):=n_h^+(\Gamma)-n_h^-(\Gamma)$) 
and $n_0^{\pm}(\Gamma):=\sum_{i=0}^{2g}n_0^{\pm}(i)(\Gamma)$. 

\begin{defn}\label{sigma} 
The number 
\[
\sigma(\Gamma):=
-6\, n_C(\Gamma)-n_L(\Gamma)+\sum_{h=1}^{[g/2]}(4h(h+1)-1)\, n_h(\Gamma)
\]
is called the {\it signature} of $\Gamma$. 
\end{defn}

Let $f\co M\rightarrow B$ be a Lefschetz fibration of genus $g$ and $\Gamma$ 
a chart in $B$ corresponding to $f$. 
The purpose of this subsection is to show the following theorem. 

\begin{thm}\label{signature} 
The signature $\sigma(M)$ of $M$ is equal to $\sigma(\Gamma)$. 
\end{thm}

\begin{rem} 
It immediately follows from Theorem \ref{signature} that 
$\sigma(\Gamma)$ is invariant under chart moves of type W and 
chart moves of transition. 
Although any combinatorial proof of this fact does not seem to be known, 
Hasegawa \cite{Hasegawa2006} proved that $\sigma(\Gamma)$ is invariant 
under chart moves of transitions by a purely combinatorial method 
on the assumption that it is invariant under chart moves of type W. 
\end{rem}


Let $\tilde{\mathcal{X}}$ be the set of right-handed Dehn twists along simple closed curves 
in $\Sigma_g$ and $\tilde{\mathcal{R}}$ the set of words in 
$\tilde{\mathcal{X}}\cup\tilde{\mathcal{X}}^{-1}$ representing an element of 
the kernel of the natural epimorphism 
from the free group generated by $\tilde{\mathcal{X}}$ to $\mathcal{M}_g$.

\begin{defn} 
For a word $w=\alpha_1\cdots \alpha_n\in\tilde{\mathcal{R}}$, we define an integer 
\[
I_g(w):=-\sum_{j=1}^{n-1} \tau_g
(\overline{\alpha_{n-j}},\overline{\alpha_{n-j+1}}\cdots\overline{\alpha_n})
-s(w), 
\]
where $\tau_g$ is the signature cocycle (Definition \ref{cocycle}), 
$\overline{\alpha}$ is the image of $\alpha\in\tilde{\mathcal{X}}\cup\tilde{\mathcal{X}}^{-1}$ 
under the composition of the natural map 
$\tilde{\mathcal{X}}\cup\tilde{\mathcal{X}}^{-1}\rightarrow\mathcal{M}_g$ and 
a natural epimorphism $\mathcal{M}_g^*\rightarrow {\rm Sp}(2g,\Bbb{Z})$, 
and $s(w)$ is the number of Dehn twists along separating simple closed curves 
included in $w$. 
\end{defn}

Suppose that $B$ is a $2$--sphere. 
If we choose a monodromy representation $\rho$ and 
a Hurwitz arc system $\mathcal{A}$ for $\Delta$ with base point $b_0$, 
we have a Hurwitz system $(\alpha_1,\ldots ,\alpha_n)\in (\mathcal{M}_g)^n$ of $f$. 
Since $\alpha_1,\ldots ,\alpha_n$ are Dehn twists and $\alpha_1\cdots \alpha_n=1$ 
in $\mathcal{M}_g$, we think $(\alpha_1,\ldots ,\alpha_n)$ as a word 
$w:=\alpha_1\cdots \alpha_n$ in $\tilde{\mathcal{R}}$. 
Theorem \ref{meyer} and Novikov's additivity for signature imply the next theorem. 

\begin{thm}[Endo and Nagami \cite{EN2004}]\label{nagami}
The signature $\sigma(M)$ of $M$ is equal to $I_g(w)$. 
\end{thm}

We are now ready to prove Theorem \ref{signature}. 

\begin{proof}[Proof of Theorem \ref{signature}] 
Choose a base point $b_0\in B-\Gamma$ 
and a disk $D$ in $B-\Gamma$ centered at $b_0$. 
We denote the set of edges of $\Gamma$ by $E(\Gamma)$. 
For each $e\in E(\Gamma)$, we choose a point $b_e$ in a region of $B-\Gamma$ 
adjacent to $e$, 
and a simple path $\gamma_e$ from $b_e$ to $b_0$ 
which intersects with edges of $\Gamma$ 
transversely and does not intersect with vertices of $\Gamma$. 
Let $w_e$ be the intersection word of $\gamma_e$ with respect to $\Gamma$ and 
$i_e\in\{0,1,\ldots ,2g\}$ the label of $e$. 
We choose a family $\{D_e\}_{e\in E(\Gamma)}$ of mutually disjoint disks included in $D$ 
and put the chart $\Gamma_e$ depicted in Figure \ref{chartE} in $D_e$ for each $e$. 

\begin{figure}[ht!]
\labellist
\footnotesize \hair 2pt
\pinlabel $w_e^{-1}$ [r] at 0 34
\pinlabel $i_e$ [t] at 62 31
\endlabellist
\centering
\includegraphics[scale=0.8]{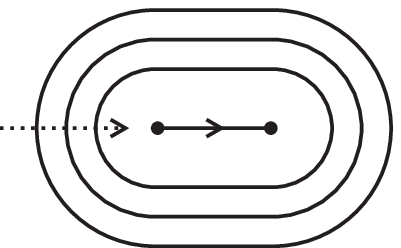}
\caption{Chart $\Gamma_e$}
\label{chartE}
\end{figure}

Taking the union of $\Gamma$ with $\Gamma_e$ for all $e\in E(\Gamma)$, 
we obtain a new chart $\Gamma_1$ in $B$, 
which describes a fiber sum $f_1\co M_1\rightarrow B$ 
of $f$ with Lefschetz fibrations over $S^2$ described by a free edge. 
For each $e\in E(\Gamma)$, we apply channel changes 
as in Figure \ref{channelC} to let a free edge 
pass through the edges intersecting with $\gamma_e$. 
We then apply a channel change as in Figure \ref{channelD} to `cut' $e$ into two edges. 
Thus we obtain a new chart $\Gamma_2$ in $B$. 

\begin{figure}[ht!]
\labellist
\footnotesize \hair 2pt
\pinlabel $e$ [b] at 1 199
\pinlabel $b_e$ [b] at 16 158
\pinlabel $i_e$ [l] at 7 128
\pinlabel $\gamma_e$ [b] at 40 160
\pinlabel $w_e$ [tl] at 85 131
\pinlabel $w_e^{-\negthinspace 1}$ [b] at 111 157
\pinlabel $b_0$ [b] at 90 158
\pinlabel $i_e$ [t] at 172 152
\pinlabel $\Gamma_e$ [l] at 230 156
\pinlabel $e$ [b] at 1 84
\pinlabel $i_e$ [l] at 7 13
\pinlabel $w_e$ [tl] at 84 16
\pinlabel $i_e$ [t] at 172 38
\pinlabel $w_e^{-\negthinspace 1}$ [b] at 172 85
\endlabellist
\centering
\includegraphics[scale=0.8]{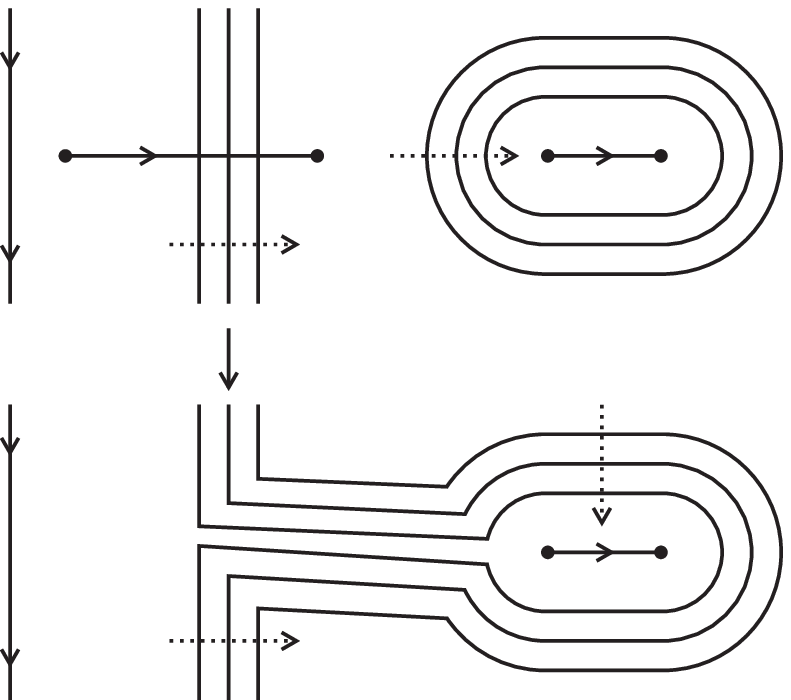}
\caption{Channel change}
\label{channelC}
\end{figure}

\begin{figure}[ht!]
\labellist
\footnotesize \hair 2pt
\pinlabel $e$ [b] at 3 86
\pinlabel $i_e$ [l] at 7 13
\pinlabel $i_e$ [l] at 70 41
\pinlabel $i_e$ [l] at 176 13
\pinlabel $i_e$ [l] at 176 72
\endlabellist
\centering
\includegraphics[scale=0.8]{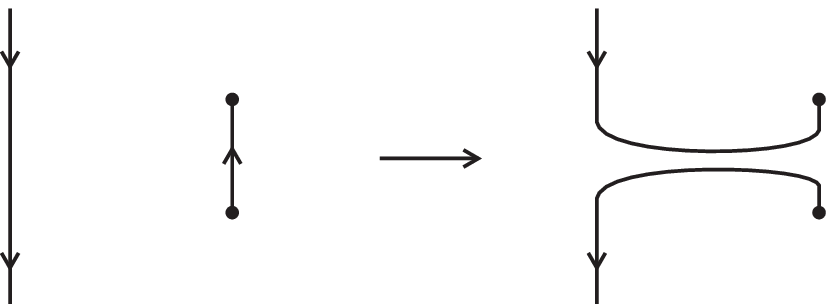}
\caption{Channel change}
\label{channelD}
\end{figure}

Since each component of $\Gamma_2$ is a tree, 
a Lefschetz fibration $f_2\co M_2\rightarrow B$ 
corresponding to $\Gamma_2$ is a fiber sum of a Lefschetz fibration 
$f_3\co M_3\rightarrow S^2$ 
with a trivial $\Sigma_g$--bundle over $B$. 
Drawing a copy of $\Gamma_2$ in $S^2$, we have a chart $\Gamma_3$ corresponding to 
$f_3$. 
The signature of a Lefschetz fibration over $S^2$ described by a free edge is equal to zero 
because $\tau_g(A,A^{-1})=0$ for any $A\in {\rm Sp}(2g,\Bbb{Z})$ 
(see Meyer \cite[Section 2]{Meyer1973}). 
Hence we have 
\[
\sigma(M)=\sigma(M_1)=\sigma(M_2)=\sigma(M_3)+\sigma(\Sigma_g\times B)
=\sigma(M_3)
\]
by Theorem \ref{classification} and Novikov's additivity. 
Since we did not change the numbers of white vertices and black vertices of type $\ell_h^{\pm}$ 
to make $\Gamma_3$ from $\Gamma$, 
we see $\sigma(\Gamma_3)=\sigma(\Gamma)$. 
Hence we only have to show $\sigma(M_3)=\sigma(\Gamma_3)$ 
in order to conclude $\sigma(M)=\sigma(\Gamma)$. 

Applying chart moves of transition to each component of $\Gamma_3$ 
as in Figure \ref{transition}, 
we remove white vertices of type $r_F(i,j)^{\pm 1}, r_B(i)^{\pm 1},r_H^{\pm 1}$ 
to obtain a union of copies of 
$L_0(i), L_h, L_h^*, R_C, R_C^*, R_L, R_L^*$, 
where $L_0(i), L_h, R_C, R_L$ 
are charts depicted in Figure \ref{chartB} and Figure \ref{chartF}, 
and $L_h^*$ (resp. $R_C^*$, 
$R_L^*$) is the mirror image of $L_h$ 
(resp. $R_C$, $R_L$) with edges orientation reversed. 
For the proof of $\sigma(M_3)=\sigma(\Gamma_3)$, it is enough to show 
that the signature of a Lefschetz fibration described by each of these charts coincides 
with the signature of the chart. 

\begin{figure}[ht!]
\labellist
\footnotesize \hair 2pt
\pinlabel $j$ [br] at 34 192
\pinlabel $i$ [bl] at 64 192
\pinlabel $i$ [tr] at 34 161
\pinlabel $j$ [tl] at 64 161
\pinlabel $j$ [br] at 103 192
\pinlabel $i$ [bl] at 134 192
\pinlabel $j$ [tl] at 134 161
\pinlabel $i$ [br] at 172 195
\pinlabel $i\negthinspace +\negthinspace 1$ [bl] at 209 195
\pinlabel $i\negthinspace +\negthinspace 1$ [r] at 165 177
\pinlabel $i$ [l] at 216 177
\pinlabel $i$ [tr] at 171 157
\pinlabel $i\negthinspace +\negthinspace 1$ [tl] at 209 157
\pinlabel $i$ [br] at 273 195
\pinlabel $i\negthinspace +\negthinspace 1$ [bl] at 309 195
\pinlabel $i\negthinspace +\negthinspace 1$ [t] at 267 175
\pinlabel $i$ [l] at 317 177
\pinlabel $i\negthinspace +\negthinspace 1$ [tl] at 310 157
\pinlabel $2g$ [b] at 43 125
\pinlabel $1$ [b] at 69 125
\pinlabel $1$ [b] at 78 125
\pinlabel $2g$ [b] at 103 125
\pinlabel $u^{-\negthinspace 1}$ [l] at 144 81
\pinlabel $1$ [l] at 142 62
\pinlabel $u$ [l] at 144 44
\pinlabel $2g$ [t] at 103 0
\pinlabel $1$ [t] at 78 0
\pinlabel $1$ [t] at 69 0
\pinlabel $2g$ [t] at 43 0
\pinlabel $u^{-\negthinspace 1}$ [r] at 2 44
\pinlabel $1$ [r] at 6 63
\pinlabel $u$ [r] at 2 80
\pinlabel $1$ [b] at 210 64
\pinlabel $u$ [b] at 238 119
\pinlabel $2g$ [b] at 255 114
\pinlabel $1$ [b] at 281 114
\pinlabel $1$ [b] at 290 114
\pinlabel $2g$ [b] at 316 114
\pinlabel $u^{-\negthinspace 1}$ [b] at 332 119
\endlabellist
\centering
\includegraphics[scale=0.8]{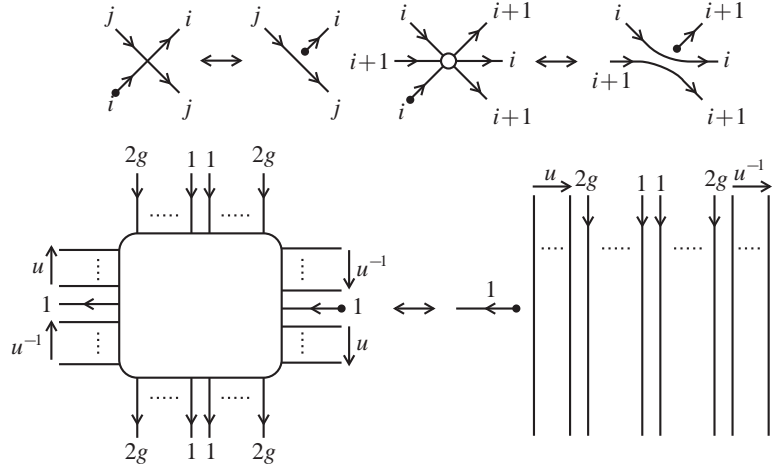}
\caption{Chart moves of transition}
\label{transition}
\end{figure}

\begin{figure}[ht!]
\labellist
\footnotesize \hair 2pt
\pinlabel $i$ [b] at 25 28
\pinlabel $1$ [b] at 218 50
\pinlabel $2h$ [b] at 199 50
\pinlabel $1$ [b] at 188 50
\pinlabel $2h$ [b] at 169 50
\pinlabel $1$ [b] at 135 50
\pinlabel $2h$ [b] at 113 50
\endlabellist
\centering
\includegraphics[scale=0.8]{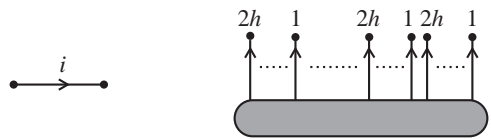}
\caption{Charts $L_0(i)$ and $L_h$}
\label{chartB}
\end{figure}

\begin{figure}[ht!]
\labellist
\footnotesize \hair 2pt
\pinlabel $1$ [b] at 124 83
\pinlabel $2$ [b] at 115 83
\pinlabel $3$ [b] at 107 83
\pinlabel $1$ [b] at 98 83
\pinlabel $2$ [b] at 91 83
\pinlabel $3$ [b] at 82 83
\pinlabel $1$ [b] at 73 83
\pinlabel $2$ [b] at 64 83
\pinlabel $3$ [b] at 56 83
\pinlabel $1$ [b] at 48 83
\pinlabel $2$ [b] at 39 83
\pinlabel $3$ [b] at 31 83
\pinlabel $0$ [t] at 9 5
\pinlabel $4$ [t] at 18 5
\pinlabel $3$ [t] at 27 5
\pinlabel $2$ [t] at 35 5
\pinlabel $1$ [t] at 43 5
\pinlabel $1$ [t] at 52 5
\pinlabel $2$ [t] at 61 5
\pinlabel $3$ [t] at 69 5
\pinlabel $4$ [t] at 78 5
\pinlabel $0$ [t] at 86 5
\pinlabel $4$ [t] at 94 5
\pinlabel $3$ [t] at 103 5
\pinlabel $2$ [t] at 112 5
\pinlabel $1$ [t] at 120 5
\pinlabel $1$ [t] at 129 5
\pinlabel $2$ [t] at 138 5
\pinlabel $3$ [t] at 146 5
\pinlabel $4$ [t] at 154 5
\pinlabel $\delta_3$ [b] at 216 88
\pinlabel $1$ [b] at 234 83
\pinlabel $3$ [b] at 243 83
\pinlabel $5$ [b] at 251 83
\pinlabel $\tau_1$ [b] at 267 88
\pinlabel $\tau_2$ [b] at 293 88
\pinlabel $0$ [t] at 315 4
\pinlabel $\tau_2^{-\negthinspace 1}$ [t] at 297 1
\pinlabel $\tau_1^{-\negthinspace 1}$ [t] at 273 1
\pinlabel $\tau_2$ [t] at 248 1
\pinlabel $0$ [t] at 230 4
\pinlabel $\tau_2^{-\negthinspace 1}$ [t] at 214 1
\pinlabel $0$ [t] at 196 4
\endlabellist
\centering
\includegraphics[scale=0.8]{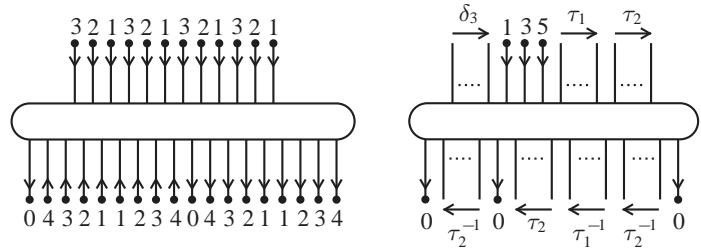}
\caption{Charts $R_C$ and $R_L$}
\label{chartF}
\end{figure}

Let $\Gamma_4$ be one of $L_0(i), L_h, L_h^*, 
R_C, R_C^*, R_L, R_L^*$ drawn in $S^2$ 
and $f_4\co M_4\rightarrow S^2$ a Lefschetz fibration described by $\Gamma_4$. 
If $\Gamma_4$ is equal to $L_0(i)$, it is easily seen that 
$\sigma(M_4)=\sigma(\Gamma_4)$. 
If $\Gamma_4$ is equal to $L_h$, 
the word $\ell_h^{-1}\sigma_h$ corresponds to a Hurwitz system of $f_4$ 
(see Figure \ref{chartB}), 
where $\sigma_h$ is a right-handed Dehn twist 
along the curve $s_h$ depicted in Figure \ref{curves}. 
Thus we have 
\[
\sigma(M_4)=I_g(\ell_h^{-1}\sigma_h)=4h(h+1)-1=\sigma(\Gamma_4)
\]
from Definition \ref{sigma}, Theorem \ref{nagami}, 
and explicit computations for $I_g$ due to Endo and Nagami 
\cite[Lemma 3.5, Proposition 3.9]{EN2004}. 
If $\Gamma_4$ is equal to $R_C$ (resp. $R_L$), 
the word $r_C$ (resp. $r_L$) corresponds to a Hurwitz system of $f_4$ 
(see Figure \ref{chartF}). 
Thus we have 
\[
\sigma(M_4)=I_g(r_C)=-6=\sigma(\Gamma_4) 
\quad 
({\rm resp.} \;\sigma(M_4)=I_g(r_L)=-1=\sigma(\Gamma_4))
\]
from Definition \ref{sigma}, Theorem \ref{nagami}, 
and formulae of Endo and Nagami 
\cite[Lemma 3.5, Remark 3.7, Propositions 3.9 and 3.10]{EN2004}. 
Suppose that $\Gamma_4$ is equal to one of $L_h^*, R_C^*, R_L^*$. 
The mirror image $\Gamma_4^*$ of $\Gamma_4$ with edges orientation reversed 
corresponds to the Lefschetz fibration $f_4\co -M_4\rightarrow S^2$
with total space orientation reversed. 
Hence we have 
\[
\sigma(M_4)=-\sigma(-M_4)=-\sigma(\Gamma_4^*)=\sigma(\Gamma_4)
\]
because we have already shown that $\sigma(M_4)=\sigma(\Gamma_4)$ 
is valid for $\Gamma_4=L_h, R_C, R_L$. 
This completes the proof of Theorem \ref{signature}. 
\end{proof}


\section{Stabilization theorems}


In this section we prove two theorems on stabilization of Lefschetz fibrations 
under taking fiber sums with copies of a fixed Lefschetz fibration. 

Following Auroux \cite{Auroux2005}, 
we first introduce a notion of universality for Lefschetz fibrations. 
Suppose that $g$ is greater than two. 

\begin{defn} 
A Lefschetz fibration of genus $g$ over $S^2$ is called {\it universal} if 
it is irreducible, chiral, and 
it contains $2g+1$ singular fibers of type ${\rm I}^+$ whose vanishing cycles 
$a_0,a_1,\ldots ,a_{2g}\subset \Sigma_g$ satisfies the following conditions: 
(i) $a_i$ and $a_{i+1}$ intersect transversely at one point for every $i\in\{1,\ldots ,2g-1\}$; 
(ii) $a_0$ and $a_4$ intersect transversely at one point; 
(iii) $a_i$ and $a_j$ does not intersect for other pairs $(i,j)$. 
A Lefschetz fibration over $S^2$ 
is universal if and only if it is described by a chart $\Gamma_0$ depicted in Figure \ref{universal} 
by virtue of Proposition \ref{correspondence}, 
where the blank labeled with ${\rm T}_0$ is filled only with edges, white vertices, 
and black vertices of type $\ell_0(i)$. 
\end{defn}

\begin{figure}[ht!]
\labellist
\footnotesize \hair 2pt
\pinlabel $0$ [b] at 16 57
\pinlabel $1$ [b] at 25 57
\pinlabel $2g$ [b] at 53 57
\pinlabel $T_0$ at 36 13
\pinlabel {The order of edges is arbitrary.} [l] at 75 45
\endlabellist
\centering
\includegraphics[scale=0.8]{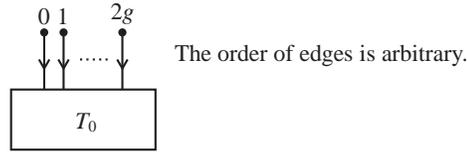}
\caption{Universal chart $\Gamma_0$}
\label{universal}
\end{figure}

\begin{rem} 
A universal Lefschetz fibration does exist for every $g$ greater than two. 
For example, Lefschetz fibrations $f_g^0, \, f_g^A, \, f_g^B, \, f_g^C, \, f_g^D$ constructed 
by Auroux \cite{Auroux2005} are universal except $f_g^D$ for $g=3$. 
There would be many universal Lefschetz fibrations of genus $g$ for a fixed $g$. 
\end{rem}

We now state the first of our main theorems. 
Let $B$ be a connected closed oriented surface and $f_0\co M_0\rightarrow S^2$ 
a universal Lefschetz fibration of genus $g$. 

\begin{thm}\label{main1}
Let $f\co M\rightarrow B$ and $f'\co M'\rightarrow B$ be Lefschetz fibrations of genus $g$. 
There exists a non-negative integer $N$ such that $f\# Nf_0$ is isomorphic to 
$f'\# Nf_0$ if and only if the following conditions hold: 
(i) $n_0^{\pm}(f)=n_0^{\pm}(f')$; 
(ii) $n_h^{\pm}(f)=n_h^{\pm}(f')$ for every $h=1,\ldots ,[g/2]$; 
(iii) $\sigma(M)=\sigma(M')$. 
\end{thm}

\begin{rem} 
Auroux \cite{Auroux2005} proved the `if' part of Theorem \ref{main1} 
for chiral Lefschetz fibrations over $S^2$ under the assumption 
that $f$ and $f'$ have sections with the same self-intersection number. 
Hasegawa \cite{Hasegawa2006} gave another proof of Auroux's theorem 
by using chart description. 
Moreover he removed the assumption about existence 
and self-intersection number of sections in Auroux's theorem. 
\end{rem}

\begin{rem} 
The isomorphism class of a fiber sum $f\#_{\Psi}\, f_0$ of a Lefschetz fibration $f$ 
with a universal Lefschetz fibration $f_0$ does not depend on a choice of 
an orientation preserving diffeomorphism $\Psi$ 
(see Proof of Theorem \ref{main1}). 
\end{rem}

\begin{proof}[Proof of Theorem \ref{main1}] 
We first prove the `if' part. Assume that $f$ and $f'$ satisfy the conditions (i), (ii), and (iii). 
Let $\Gamma$ and $\Gamma'$ be charts in $B$ corresponding to $f$ and $f'$, 
respectively. 
We suppose that 
$f_0$ is described by a chart $\Gamma_0$ depicted in Figure \ref{universal}. 
Since every edge has two adjacent vertices, 
the sum of the signed numbers of adjacent edges for all vertices of $\Gamma$ 
is equal to zero: 
\[
10n_C(\Gamma)+n_L(\Gamma)-\sum_{i=0}^{2g}n_0(i)(\Gamma)
-4\sum_{h=1}^{[g/2]}h(2h+1)\cdot n_h(\Gamma)=0. 
\]
A similar equality for $\Gamma'$ also holds. 
Interpreting the conditions (i) and (ii) as conditions on $\Gamma$ and $\Gamma'$, 
we have $\sum_{i=0}^{2g}n_0(i)(\Gamma)=\sum_{i=0}^{2g}n_0(i)(\Gamma')$ 
and $n_h(\Gamma)=n_h(\Gamma')$ for $h=1,\ldots ,[g/2]$. 
Thus we obtain 
\[
10n_C(\Gamma)+n_L(\Gamma)=10n_C(\Gamma')+n_L(\Gamma'). 
\]
On the other hand, we have 
\[
-6\, n_C(\Gamma)-n_L(\Gamma)=-6\, n_C(\Gamma')-n_L(\Gamma')
\]
by the condition (iii), Theorem \ref{signature}, and 
$n_h(\Gamma)=n_h(\Gamma')$ for $h=1,\ldots ,[g/2]$. 
Hence $n_C(\Gamma)=n_C(\Gamma')$ and $n_L(\Gamma)=n_L(\Gamma')$. 

Let $N$ be an integer larger than both of the number of edges of $\Gamma$ 
and that of $\Gamma'$. 
Choose a base point $b_0\in B-(\Gamma\cup\Gamma')$. 
The fiber sum $f\# Nf_0$ is described by a chart 
$(\cdots ((\Gamma\#_{w_1}\Gamma_0)\#_{w_2}\Gamma_0)\cdots )\#_{w_N}\Gamma_0$ 
for some words $w_1,\ldots ,w_N$ in $\mathcal{X}\cup\mathcal{X}^{-1}$. 
Since hoops surrounding $\Gamma_0$ can be removed by use of the edges of $\Gamma_0$ 
as in Figure \ref{removeA}, 
the chart is transformed into a product $\Gamma\oplus N\Gamma_0$ 
by channel changes. Similarly, the fiber sum $f'\# Nf_0$ is described by a product 
$\Gamma'\oplus N\Gamma_0$. 

\begin{figure}[ht!]
\labellist
\footnotesize \hair 2pt
\pinlabel $i$ [b] at 41 71
\pinlabel $0$ [b] at 23 51
\pinlabel $i$ [b] at 43 52
\pinlabel $2g$ [b] at 61 51
\pinlabel $T_0$ at 43 11
\pinlabel $i$ [b] at 188 72
\pinlabel $0$ [b] at 142 51
\pinlabel $2g$ [b] at 179 51
\pinlabel $T_0$ at 161 11
\pinlabel $0$ [b] at 257 51
\pinlabel $i$ [b] at 275 52
\pinlabel $2g$ [b] at 293 51
\pinlabel $T_0$ at 276 11
\endlabellist
\centering
\includegraphics[scale=0.8]{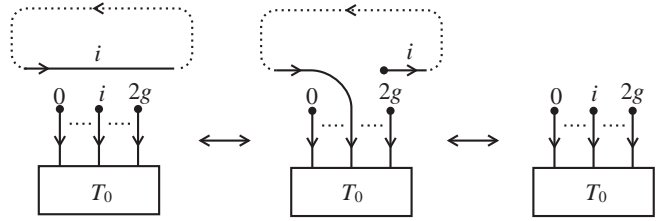}
\caption{Removing a hoop}
\label{removeA}
\end{figure}

We choose and fix $2g+1$ edges of $\Gamma_0$ which are 
labeled with $0,1,\ldots ,2g$ and adjacent to black vertices. 
We apply chart moves only to these edges in the following. 
Since $\Gamma_0$ can pass through any edge of $\Gamma$ as shown in Figure \ref{pass}, 
we can move $\Gamma_0$ to any region of $B-\Gamma$ by channel changes. 
\begin{figure}[ht!]
\labellist
\footnotesize \hair 2pt
\pinlabel $i$ [b] at 35 69
\pinlabel $0$ [b] at 18 51
\pinlabel $i$ [b] at 36 52
\pinlabel $2g$ [b] at 55 51
\pinlabel $T_0$ at 35 11
\pinlabel {(a)} [t] at 37 0
\pinlabel $i$ [b] at 106 72
\pinlabel $i$ [b] at 160 72
\pinlabel $0$ [b] at 115 52
\pinlabel $2g$ [b] at 151 52
\pinlabel $T_0$ at 134 11
\pinlabel {(b)} [t] at 134 0
\pinlabel $i$ [t] at 200 66
\pinlabel $i$ [t] at 255 66
\pinlabel $2g$ [t] at 210 86
\pinlabel $0$ [t] at 246 86
\pinlabel $T_0$ at 228 126
\pinlabel {(c)} [t] at 228 0
\pinlabel $i$ [t] at 322 66
\pinlabel $2g$ [t] at 303 86
\pinlabel $i$ [t] at 321 85
\pinlabel $0$ [t] at 341 86
\pinlabel $T_0$ at 322 126
\pinlabel {(d)} [t] at 321 0
\endlabellist
\centering
\includegraphics[scale=0.8]{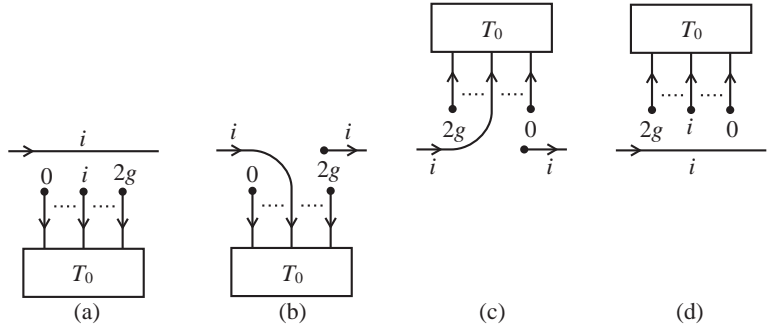}
\caption{Passing through an edge}
\label{pass}
\end{figure}
For each edge of $\Gamma$, we move a copy of $\Gamma_0$ to a region adjacent to 
the edge and apply a channel change to the edge and $\Gamma_0$ 
as in Figure \ref{pass} (a) and (b). 
Applying chart moves of transition to each component of the chart 
as in Figure \ref{transition}, 
we remove white vertices of type $r_F(i,j)^{\pm 1}, r_B(i)^{\pm 1},r_H^{\pm 1}$ 
to obtain a union of copies of 
$L_0(i), \tilde{L}_h, L_h^*, \tilde{R}_C, \hat{R}_C, 
\tilde{R}_L, \hat{R}_L, \Gamma_0$ shown in 
Figures \ref{chartH}, \ref{chartC}, \ref{chartD}, 
where we use a simplification of diagrams as in Figure \ref{symbol}. 

\begin{figure}[ht!]
\labellist
\footnotesize \hair 2pt
\pinlabel $1$ [b] at 113 50
\pinlabel $2h$ [b] at 92 50
\pinlabel $1$ [b] at 82 50
\pinlabel $2h$ [b] at 64 50
\pinlabel $1$ [b] at 30 50
\pinlabel $2h$ [b] at 8 50
\endlabellist
\centering
\includegraphics[scale=0.8]{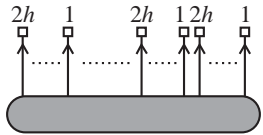}
\caption{Chart $\tilde{L}_h$}
\label{chartH}
\end{figure}

\begin{figure}[ht!]
\labellist
\footnotesize \hair 2pt
\pinlabel $1$ [b] at 124 78
\pinlabel $2$ [b] at 115 78
\pinlabel $3$ [b] at 107 78
\pinlabel $1$ [b] at 98 78
\pinlabel $2$ [b] at 91 78
\pinlabel $3$ [b] at 82 78
\pinlabel $1$ [b] at 73 78
\pinlabel $2$ [b] at 64 78
\pinlabel $3$ [b] at 56 78
\pinlabel $1$ [b] at 48 78
\pinlabel $2$ [b] at 39 78
\pinlabel $3$ [b] at 31 78
\pinlabel $0$ [t] at 9 0
\pinlabel $4$ [t] at 18 0
\pinlabel $3$ [t] at 27 0
\pinlabel $2$ [t] at 35 0
\pinlabel $1$ [t] at 43 0
\pinlabel $1$ [t] at 52 0
\pinlabel $2$ [t] at 61 0
\pinlabel $3$ [t] at 69 0
\pinlabel $4$ [t] at 78 0
\pinlabel $0$ [t] at 86 0
\pinlabel $4$ [t] at 94 0
\pinlabel $3$ [t] at 103 0
\pinlabel $2$ [t] at 112 0
\pinlabel $1$ [t] at 120 0
\pinlabel $1$ [t] at 129 0
\pinlabel $2$ [t] at 138 0
\pinlabel $3$ [t] at 146 0
\pinlabel $4$ [t] at 154 0
\pinlabel $1$ [b] at 226 79
\pinlabel $2$ [b] at 235 79
\pinlabel $3$ [b] at 244 79
\pinlabel $1$ [b] at 252 79
\pinlabel $2$ [b] at 260 79
\pinlabel $3$ [b] at 269 79
\pinlabel $1$ [b] at 278 79
\pinlabel $2$ [b] at 286 79
\pinlabel $3$ [b] at 295 79
\pinlabel $1$ [b] at 302 79
\pinlabel $2$ [b] at 311 79
\pinlabel $3$ [b] at 320 79
\pinlabel $0$ [t] at 341 0
\pinlabel $4$ [t] at 333 0
\pinlabel $3$ [t] at 324 0
\pinlabel $2$ [t] at 316 0
\pinlabel $1$ [t] at 308 0
\pinlabel $1$ [t] at 299 0
\pinlabel $2$ [t] at 291 0
\pinlabel $3$ [t] at 281 0
\pinlabel $4$ [t] at 274 0
\pinlabel $0$ [t] at 265 0
\pinlabel $4$ [t] at 256 0
\pinlabel $3$ [t] at 247 0
\pinlabel $2$ [t] at 239 0
\pinlabel $1$ [t] at 231 0
\pinlabel $1$ [t] at 222 0
\pinlabel $2$ [t] at 214 0
\pinlabel $3$ [t] at 206 0
\pinlabel $4$ [t] at 197 0
\endlabellist
\centering
\includegraphics[scale=0.8]{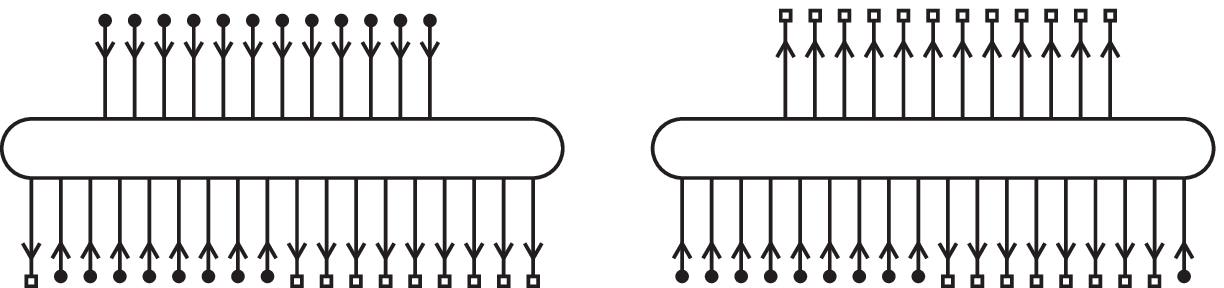}
\caption{Charts $\tilde{R}_C$ and $\hat{R}_C$}
\label{chartC}
\end{figure}

\begin{figure}[ht!]
\labellist
\footnotesize \hair 2pt
\pinlabel $\delta_3$ [b] at 29 88
\pinlabel $1$ [b] at 46 83
\pinlabel $3$ [b] at 55 83
\pinlabel $5$ [b] at 64 83
\pinlabel $\tau_1$ [b] at 80 88
\pinlabel $\tau_2$ [b] at 105 88
\pinlabel $0$ [t] at 128 3
\pinlabel $\tau_2^{-\negthinspace 1}$ [t] at 110 0
\pinlabel $\tau_1^{-\negthinspace 1}$ [t] at 85 0
\pinlabel $\tau_2$ [t] at 60 0
\pinlabel $0$ [t] at 43 3
\pinlabel $\tau_2^{-\negthinspace 1}$ [t] at 26 0
\pinlabel $0$ [t] at 9 3
\pinlabel $0$ [t] at 296 3
\pinlabel $\tau_2$ [t] at 279 0
\pinlabel $0$ [t] at 262 3
\pinlabel $\tau_2^{-\negthinspace 1}$ [t] at 245 0
\pinlabel $\tau_1$ [t] at 220 0
\pinlabel $\tau_2$ [t] at 195 0
\pinlabel $0$ [t] at 177 3
\pinlabel $\tau_2^{-\negthinspace 1}$ [b] at 197 88
\pinlabel $\tau_1^{-\negthinspace 1}$ [b] at 223 88
\pinlabel $5$ [b] at 240 85
\pinlabel $3$ [b] at 249 85
\pinlabel $1$ [b] at 258 85
\pinlabel $\delta_3^{-\negthinspace 1}$ [b] at 273 88
\endlabellist
\centering
\includegraphics[scale=0.8]{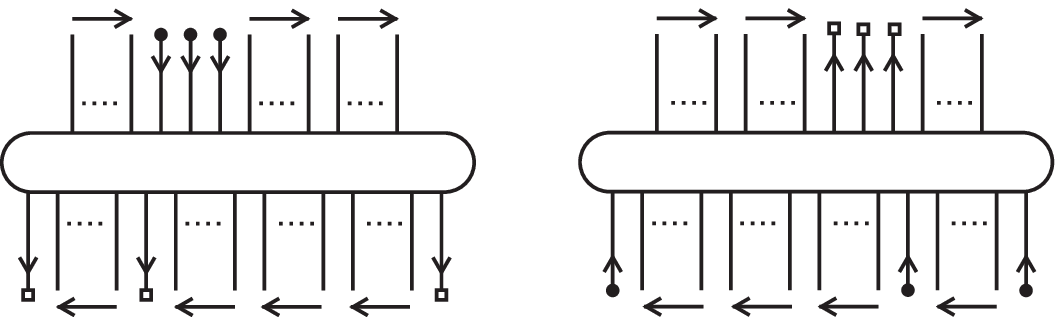}
\caption{Charts $\tilde{R}_L$ and $\hat{R}_L$}
\label{chartD}
\end{figure}

\begin{figure}[ht!]
\labellist
\footnotesize \hair 2pt
\pinlabel $i$ [b] at 3 45
\pinlabel $0$ [b] at 89 57
\pinlabel $i$ [b] at 107 71
\pinlabel $2g$ [b] at 126 57
\pinlabel $T_0$ at 108 14
\pinlabel $=$ at 40 28
\endlabellist
\centering
\includegraphics[scale=0.8]{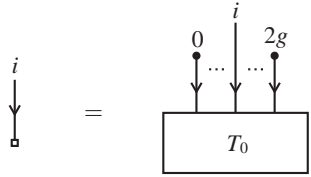}
\caption{Simplification of diagram}
\label{symbol}
\end{figure}

If there is a pair of $\tilde{R}_C$ and $\hat{R}_C$, 
we remove them by a death of a pair of white vertices to obtain many copies of $\Gamma_0$. 
Similarly, we remove a pair of $\tilde{R}_L$ and $\hat{R}_L$. 
Since there is at least one $\Gamma_0$, any copy of $L_0(i)$ can be transformed into 
$L_0(1)$ as in Figure \ref{label}. 

\begin{figure}[ht!]
\labellist
\footnotesize \hair 2pt
\pinlabel $j$ [l] at 28 152
\pinlabel $0$ [b] at 83 186
\pinlabel $i$ [b] at 84 201
\pinlabel $2g$ [b] at 120 186
\pinlabel $T_0$ at 101 143
\pinlabel $j$ [l] at 199 152
\pinlabel $i$ [l] at 212 164
\pinlabel $0$ [b] at 255 186
\pinlabel $j$ [b] at 255 201
\pinlabel $2g$ [b] at 292 186
\pinlabel $T_0$ at 274 143
\pinlabel $j$ [l] at 370 152
\pinlabel $i$ [l] at 382 150
\pinlabel $j$ [l] at 390 142
\pinlabel $0$ [b] at 426 186
\pinlabel $j$ [b] at 444 186
\pinlabel $2g$ [b] at 463 186
\pinlabel $T_0$ at 444 143
\pinlabel $i$ [b] at 26 91
\pinlabel $j$ [t] at 26 46
\pinlabel $i$ [b] at 26 23
\pinlabel $j$ [t] at 26 13
\pinlabel $0$ [b] at 82 73
\pinlabel $j$ [b] at 100 73
\pinlabel $2g$ [b] at 119 73
\pinlabel $T_0$ at 100 29
\pinlabel $i$ [l] at 202 82
\pinlabel $j$ [tl] at 216 8
\pinlabel $i$ [b] at 201 10
\pinlabel $0$ [b] at 254 73
\pinlabel $j$ [b] at 273 73
\pinlabel $2g$ [b] at 291 73
\pinlabel $T_0$ at 273 29
\pinlabel $i$ [l] at 372 44
\pinlabel $0$ [b] at 425 73
\pinlabel $j$ [b] at 444 73
\pinlabel $2g$ [b] at 463 73
\pinlabel $T_0$ at 443 29
\endlabellist
\centering
\includegraphics[scale=0.75]{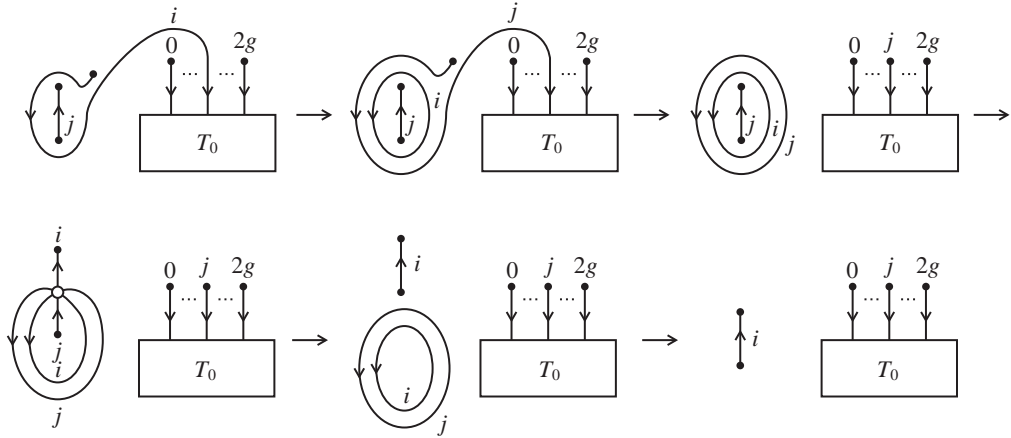}
\caption{Changing a label ($j=i+1$ or $(i,j)=(4,0)$)}
\label{label}
\end{figure}

Thus we have a union $\Gamma_1$ of 
$n_0^-(\Gamma)$ copies of $L_0(1)$, 
$n_h^+(\Gamma)$ copies of $\tilde{L}_h$, 
$n_h^-(\Gamma)$ copies of $L_h^*$, 
$|n_C(\Gamma)|$ copies of $\tilde{R}_C$ (or $\hat{R}_C$), 
$|n_L(\Gamma)|$ copies of $\tilde{R}_L$ (or $\hat{R}_L$), 
and $k$ copies of $\Gamma_0$ for some $k$. 
A similar argument implies that $\Gamma'\oplus N\Gamma_0$ is transformed into 
a union $\Gamma'_1$ of 
$n_0^-(\Gamma')$ copies of $L_0(1)$, 
$n_h^+(\Gamma')$ copies of $\tilde{L}_h$, 
$n_h^-(\Gamma')$ copies of $L_h^*$, 
$|n_C(\Gamma')|$ copies of $\tilde{R}_C$ (or $\hat{R}_C$), 
$|n_L(\Gamma')|$ copies of $\tilde{R}_L$ (or $\hat{R}_L$), 
and $k'$ copies of $\Gamma_0$ for some $k'$ 
by chart moves of type W and chart moves of transition. 
By virtue of the conditions (i) and (ii) together with $n_C(\Gamma)=n_C(\Gamma')$, 
$n_L(\Gamma)=n_L(\Gamma')$, $n_0^+(\Gamma\oplus N\Gamma_0)=n_0^+(\Gamma_1)$, 
and $n_0^+(\Gamma'\oplus N\Gamma_0)=n_0^+(\Gamma'_1)$, 
we conclude that $k=k'$ because of $n_0^+(\Gamma_0)\ne 0$. 
Hence $\Gamma_1$ is transformed into $\Gamma'_1$ by an ambient isotopy of $B$ 
relative to $b_0$, 
which means that $\Gamma\oplus N\Gamma_0$ is transformed into 
$\Gamma'\oplus N\Gamma_0$ by chart moves of type W, chart moves of transition, 
and ambient isotopies of $B$ relative to $b_0$. 
Therefore $f\# Nf_0$ is (strictly) isomorphic to $f'\# Nf_0$ by Theorem \ref{classification}. 

We next prove the `only if' part. 
Take a non-negative integer $N$ so that $f\# Nf_0$ is isomorphic to $f'\# Nf_0$. 
Since an isomorphism preserves numbers and types of vanishing cycles and signatures, 
we have $n_0^{\pm}(f\# Nf_0)=n_0^{\pm}(f'\# Nf_0)$, 
$n_h^{\pm}(f\# Nf_0)=n_h^{\pm}(f'\# Nf_0)$ for every $h=1,\ldots ,[g/2]$, 
and $\sigma(M\#_F NM_0)=\sigma(M'\#_F NM_0)$. 
The conditions (i), (ii), (iii) follows from additivity of $n_0^{\pm}, n_h^{\pm}, \sigma$ 
under fiber sum. 
\end{proof}

\begin{defn} 
A Lefschetz fibration of genus $g$ over $S^2$ is called {\it elementary} if 
it contains exactly two singular fibers of type ${\rm I}^+$ and of type ${\rm I}^-$ 
which have the same vanishing cycles. 
A chart $L_0(i)$ in $S^2$ corresponds to an elementary Lefschetz fibration. 
\end{defn}

\begin{rem} 
Two elementary Lefschetz fibrations of genus $g$ are isomorphic to each other. 
The total space of an elementary Lefschetz fibration of genus $g$ is diffeomorphic to 
$\Sigma_{g-1}\times S^2\# S^1\times S^3$. 
\end{rem}


We state the second of our main theorems. 
Let $B$ be a connected closed oriented surface and $f_{\star}\co M_{\star}\rightarrow S^2$ 
an elementary Lefschetz fibration of genus $g$. 

\begin{thm}\label{main2}
Let $f\co M\rightarrow B$ and $f'\co M'\rightarrow B$ be Lefschetz fibrations of genus $g$. 
There exists a non-negative integer $N$ such that a fiber sum $f\# Nf_{\star}$ is isomorphic to 
a fiber sum $f'\# Nf_{\star}$ if and only if the following conditions hold: 
(i) $n_0^{\pm}(f)=n_0^{\pm}(f')$; 
(ii) $n_h^{\pm}(f)=n_h^{\pm}(f')$ for every $h=1,\ldots ,[g/2]$; 
(iii) $\sigma(M)=\sigma(M')$. 
\end{thm}

\begin{rem}
In contrast to Theorem \ref{main1}, 
the isomorphism class of a fiber sum $f\#_{\Psi}\, f_{\star}$ of a Lefschetz fibration $f$ 
with an elementary Lefschetz fibration $f_{\star}$ depends on a choice of 
an orientation preserving diffeomorphism $\Psi$ in general. 
\end{rem}

\begin{proof}[Proof of Theorem \ref{main2}] 
We only show the `if' part. The `only if' part is the same as that of the proof of 
Theorem \ref{main1}. 

Assume that $f$ and $f'$ satisfy the conditions (i), (ii), and (iii). 
Let $\Gamma$ and $\Gamma'$ be charts in $B$ corresponding to $f$ and $f'$, 
respectively. 
It follows from the same argument as in the proof of Theorem \ref{main1} 
that $n_C(\Gamma)=n_C(\Gamma')$ and $n_L(\Gamma)=n_L(\Gamma')$. 
Let $N$ be an integer larger than both of the number of edges of $\Gamma$ 
and that of $\Gamma'$. 
We construct the chart $\Gamma_e$ in $B$ for each $e\in E(\Gamma)$ 
as in the proof of Theorem \ref{signature}. 
Taking the union of $\Gamma$ with $\Gamma_e$ for all $e\in E(\Gamma)$ 
and with $N-\#E(\Gamma)$ copies of $L_1(1)$, 
we obtain a new chart $\Gamma_1$ in $B$, 
which describes a fiber sum $f\# Nf_{\star}$. 
Applying channel changes as in the proof of Theorem \ref{signature} and 
deaths of pairs of white vertices appropriately, 
we obtain a union $\Gamma_2$ of 
$n_h^+(\Gamma)$ copies of $L_h$, 
$n_h^-(\Gamma)$ copies of $L_h^*$, 
$|n_C(\Gamma)|$ copies of $R_C$ (or $R_C^*$), 
$|n_L(\Gamma)|$ copies of $R_L$ (or $R_L^*$), 
and $k_i$ copies of $L_0(i)$ for some $k_i$, 
Similarly, $\Gamma'$ is transformed into $\Gamma'_1$, 
which describes a fiber sum $f'\# Nf_{\star}$, and then 
a union $\Gamma'_2$ of 
$n_h^+(\Gamma')$ copies of $L_h$, 
$n_h^-(\Gamma')$ copies of $L_h^*$, 
$|n_C(\Gamma')|$ copies of $R_C$ (or $R_C^*$), 
$|n_L(\Gamma')|$ copies of $R_L$ (or $R_L^*$), 
and $k'_i$ copies of $L_0(i)$ for some $k'_i$. 

A similar argument on the number $n_0^+$ as in the proof of Theorem \ref{main1} 
implies that $k_0+k_1+\cdots +k_{2g}=k'_0+k'_1+\cdots +k'_{2g}$. 
Adding $|k_i-k'_i|$ copies of $L_0(i)$ to either $\Gamma_2$ or $\Gamma'_2$ if necessary, 
we may assume that $k_i=k'_i$ for every $i\in\{0,1,\ldots ,2g\}$. 
Hence $\Gamma_2$ is transformed into $\Gamma'_2$ by an ambient isotopy of $B$ 
relative to $b_0$, 
which means that $f\# Nf_{\star}$ is (strictly) isomorphic to $f'\# Nf_{\star}$ 
by Theorem \ref{classification}. 
\end{proof} 

Let $g$ be an integer greater than two and 
$B_1,\ldots ,B_r$ connected closed oriented surfaces. 
We consider a Lefschetz fibration $f_i\co M_i\rightarrow B_i$ of genus $g$ 
for each $i\in\{1,\ldots ,r\}$, 
and a universal Lefschetz fibration $f_0\co M_0\rightarrow S^2$ of genus $g$. 

\begin{prop}\label{sums}
For (possibly different) fiber sums $f$ and $f'$ of $f_1,\ldots ,f_r$, 
fiber sums $f\# f_0$ and $f'\# f_0$ are isomorphic to each other. 
\end{prop}

\begin{proof}
Let $\Gamma$ and $\Gamma'$ be charts corresponding to $f$ and $f'$. 
Since hoops surrounding a component of $\Gamma$ (and $\Gamma'$)
can be removed by use of the edges of $\Gamma_0$ 
as in Figure \ref{removeA}, $\Gamma\#\Gamma_0$ and $\Gamma'\#\Gamma_0$ 
are transformed into the same chart. 
\end{proof}

\begin{rem} 
Proposition \ref{sums} implies that there are many examples of non-isomorphic 
Lefschetz fibrations with the same base, the same fiber, and the same numbers of singular 
fibers of each type which become isomorphic after one stabilization. 
For example, the Lefschetz fibration on $E(n)_K$ constructed by Fintushel and Stern 
\cite[Theorem 14]{FS2004} (see also Park and Yun \cite{PY2009}) 
for a fibered knot $K$ becomes isomorphic to that on $E(n)_{K'}$ 
for another fibered knot $K'$ of the same genus after one stabilization. 
Similar results hold for Lefschetz fibrations on $Y(n;K_1,K_2)$ constructed by 
Fintushel and Stern \cite[\S 7]{FS2004} (see also Park and Yun \cite{PY2011})
as well as fiber sums of (generalizations of) Matsumoto's fibration 
studied by Ozbagci and Stipsicz \cite{OS2000}, 
Korkmaz \cite{Korkmaz2001, Korkmaz2009}, 
and Okamori \cite{Okamori2011}. 
\end{rem}


\section{Variations and problems}


In this section we discuss possible variations of chart description for Lefschetz fibrations. 

If we replace the triple $(\mathcal{X},\mathcal{R},\mathcal{S})$ 
defined in Section 2 with other triples, 
we obtain various chart descriptions for Lefschetz fibrations 
(see Kamada \cite{Kamada2007} and Hasegawa \cite{Hasegawa2006}). 

We first choose large $\mathcal{X},\mathcal{R}$, and $\mathcal{S}$. 
Let $\mathcal{X}$ be the set of right-handed Dehn twists along simple closed curves 
in $\Sigma_g$ 
and $\mathcal{S}$ the set of Dehn twists along non-trivial simple closed curves in $\Sigma_g$. 
By virtue of a theorem of Luo \cite{Luo1997}, 
$\langle \mathcal{X}\, |\,\mathcal{R}\rangle$ gives an infinite presentation of $\mathcal{M}_g$ 
for the set $\mathcal{R}$ of the following four kinds of words: 
(0) trivial relator $r_T:=a$, where $a$ is the Dehn twist along a trivial simple closed curve 
on $\Sigma_g$; 
(1) primitive braid relator $r_P:=b^{-1}abc^{-1}$, 
where $a,b,c\in\mathcal{X}$ and 
the curve for $c$ is the image of the curve for $a$ by $b$; 
(2) $2$--chain relator $r_C:=(c_2c_1)^6d^{-1}$, where 
$c_1,c_2,d\in \mathcal{X}$ and 
the curves for $c_1$ and $c_2$ intersect transversely at one point and 
the curve for $d$ is the boundary 
curve of a regular neighborhood of the union of the curves for $c_1$ and $c_2$; 
(3) lantern relator $r_L:=cbad_4^{-1}d_3^{-1}d_2^{-1}d_1^{-1}$, 
where $a,b,c,d_1,d_2,d_3,d_4\in\mathcal{X}$ and 
the curves for $a$ and $b$ intersect transversely at two points with algebraic 
intersection number zero, the curve for 
$c$ is obtained by resolving the intersections of these two curves, 
and the curves for $d_1,d_2,d_3,d_4$ are the boundary curves of a regular neighborhood 
of those for $a,b,c$. 

Let $B$ be a connected closed oriented surface. 
Charts in $B$ for the triple $(\mathcal{X},\mathcal{R},\mathcal{S})$ defined above have 
white vertices of type 
$r_T^{\pm 1}, r_P^{\pm 1}, r_C^{\pm 1}, r_L^{\pm 1}$ (see Figure \ref{verticesD}). 
For a chart $\Gamma$ in $B$, we denote the number of white vertices of type $r_X$ 
minus the number of white vertices of type $r_X^{-1}$ included in $\Gamma$ 
by $n_X(\Gamma)$, where $X=T,P,C,L$. 

\begin{figure}[ht!]
\labellist
\footnotesize \hair 2pt
\pinlabel $a$ [b] at 4 54
\pinlabel $a$ [br] at 40 50
\pinlabel $b$ [bl] at 70 50
\pinlabel $b$ [tr] at 40 24
\pinlabel $c$ [tl] at 70 24
\pinlabel $c_2$ [b] at 114 73
\pinlabel $c_1$ [b] at 123 73
\pinlabel $c_2$ [b] at 132 73
\pinlabel $c_1$ [b] at 141 73
\pinlabel $c_2$ [b] at 150 73
\pinlabel $c_1$ [b] at 159 73
\pinlabel $c_2$ [b] at 167 73
\pinlabel $c_1$ [b] at 176 73
\pinlabel $c_2$ [b] at 185 73
\pinlabel $c_1$ [b] at 194 73
\pinlabel $c_2$ [b] at 203 73
\pinlabel $c_1$ [b] at 212 73
\pinlabel $d$ [t] at 162 0
\pinlabel $c$ [b] at 263 73
\pinlabel $b$ [b] at 271 73
\pinlabel $a$ [b] at 281 73
\pinlabel $d_1$ [t] at 257 0
\pinlabel $d_2$ [t] at 267 0
\pinlabel $d_3$ [t] at 277 0
\pinlabel $d_4$ [t] at 287 0
\endlabellist
\centering
\includegraphics[scale=0.8]{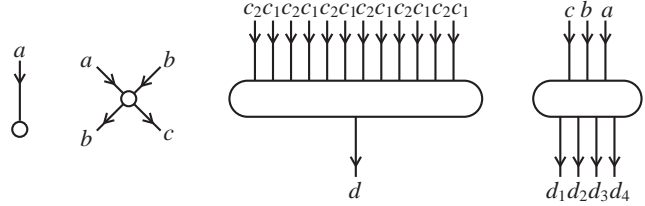}
\caption{Vertices of type $r_T$, $r_P$, $r_C$, $r_L$}
\label{verticesD}
\end{figure}

\begin{prop}\label{signature2}
The signature $\sigma(M)$ of the total space $M$ of a Lefschetz fibration 
$f\co M\rightarrow B$ described by $\Gamma$ is equal to 
$-n_T(\Gamma)-7n_C(\Gamma)+n_L(\Gamma)$. 
\end{prop}

\begin{proof}
It is seen by a similar argument to the proof of Theorem \ref{signature}. 
\end{proof}

\begin{exam} Let $B$ be a connected closed oriented surface of genus $2$ 
and $(\mathcal{X},\mathcal{R},\mathcal{S})$ the triple defined above for $g=3$. 
Let $a,b,c,d_1,d_2,d_3,d_4,c_1,c_2,c_3$ be right-handed Dehn twists along 
simple closed curves of the same names on $\Sigma_3$ depicted in Figure \ref{twists}. 
We present $B$ as an octagon with opposite sides identified and 
consider a chart $\Gamma$ and loops 
$\gamma_1,\gamma_2,\gamma_3,\gamma_4,\gamma_5$ based at $b_0$ in $B$ 
as in Figure \ref{exampleB}. 
We use a simplification of diagrams as in Figure \ref{verticesE} (a) 
if the curves for $x,y\in\mathcal{X}$ intersect transversely at one point, 
and that as in Figure \ref{verticesE} (b) if the curves for $x$ and $y$ are disjoint. 

Since the intersection words of $\gamma_1,\gamma_2,\gamma_3,\gamma_4,\gamma_5$ 
with respect to $\Gamma$ are 
\begin{align*}
w_{\Gamma}(\gamma_1) & =d_2^{-1}, \quad 
w_{\Gamma}(\gamma_2) =c_3^{-1}c^{-1}d_2^{-1}c_3^{-1}, \\
w_{\Gamma}(\gamma_3) & =c_1^{-1}b^{-1}c_2^{-1}d_3^{-1}
a^{-1}c_2^{-1}d_4^{-1}c_1^{-1}, \quad 
w_{\Gamma}(\gamma_4) =d_4a^{-1}, \quad 
w_{\Gamma}(\gamma_5) =d_1, 
\end{align*}
a Lefschetz fibration $f:M\rightarrow B$ of genus $3$ 
described by $\Gamma$ is isomorphic to 
the Lefschetz fibration constructed by Korkmaz and Ozbagci \cite[Theorem 1.2]{KO2001}. 
$f$ has only one singular fiber and it is of type ${\rm I}^+$. 
We can compute the signature $\sigma(M)$ of the total space $M$ 
by Proposition \ref{signature2}: 
\[
\sigma(M)=n_L(\Gamma)=-1, 
\]
which coincides with the value computed in \cite[Proposition 14]{EKKOS2002}. 
\end{exam}

\begin{figure}[ht!]
\labellist
\footnotesize \hair 2pt
\pinlabel $d_1$ [b] at 16 53
\pinlabel $d_4$ [b] at 70 53
\pinlabel $d_3$ [b] at 124 53
\pinlabel $d_2$ [b] at 177 53
\pinlabel $a$ [r] at 88 25
\pinlabel $b$ [t] at 132 17
\pinlabel $c_1$ [b] at 247 60
\pinlabel $c_2$ [b] at 302 60
\pinlabel $c_3$ [b] at 356 60
\pinlabel $c$ [t] at 284 16
\endlabellist
\centering
\includegraphics[scale=0.75]{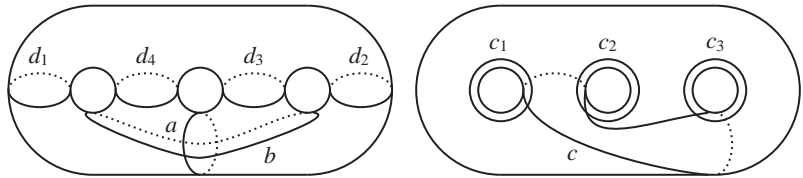}
\caption{Simple closed curves on $\Sigma_3$}
\label{twists}
\end{figure}

\begin{figure}[ht!]
\labellist
\footnotesize \hair 2pt
\pinlabel $b_0$ [t] at 234 450
\pinlabel {\small $\gamma_1$} [b] at 145 427
\pinlabel {\small $\gamma_2$} [b] at 23 305
\pinlabel {\small $\gamma_1$} [t] at 24 144
\pinlabel {\small $\gamma_2$} [t] at 146 26
\pinlabel {\small $\gamma_3$} [t] at 323 33
\pinlabel {\small $\gamma_4$} [t] at 430 142
\pinlabel {\small $\gamma_3$} [b] at 432 306
\pinlabel {\small $\gamma_4$} [b] at 310 427
\pinlabel {\small $\gamma_5$} [r] at 116 361
\pinlabel $d_2$ [b] at 188 400
\pinlabel $d_1$ [t] at 158 358
\pinlabel $c$ [r] at 153 304
\pinlabel $b$ [r] at 193 289
\pinlabel $a$ [t] at 228 328
\pinlabel $d_4$ [b] at 259 404
\pinlabel $d_3$ [b] at 230 384
\pinlabel $c_3$ [b] at 97 324
\pinlabel $c$ [b] at 73 287
\pinlabel $d_2$ [b] at 58 253
\pinlabel $c_3$ [b] at 41 224
\pinlabel $c_3$ [r] at 104 219
\pinlabel $d_2$ [b] at 59 158
\pinlabel $c_3$ [l] at 161 212
\pinlabel $c$ [r] at 135 193
\pinlabel $d_2$ [l] at 113 143
\pinlabel $c_3$ [r] at 88 113
\pinlabel $c_1$ [b] at 240 300
\pinlabel $b$ [b] at 275 295
\pinlabel $c_2$ [r] at 278 245
\pinlabel $d_3$ [l] at 298 239
\pinlabel $a$ [r] at 319 198
\pinlabel $c_2$ [l] at 342 187
\pinlabel $d_4$ [t] at 412 260
\pinlabel $c_1$ [l] at 424 223
\pinlabel $a$ [t] at 420 184
\pinlabel $c_2$ [b] at 334 315
\pinlabel $d_3$ [t] at 353 310
\pinlabel $a$ [r] at 370 289
\pinlabel $c_2$ [t] at 383 270
\pinlabel $c_2$ [t] at 343 277
\pinlabel $c_1$ [t] at 292 171
\pinlabel $c_1$ [r] at 205 183
\pinlabel $b$ [l] at 224 164
\pinlabel $d_4$ [r] at 341 102
\pinlabel $c_1$ [l] at 364 82
\pinlabel $d_4$ [t] at 388 120
\pinlabel {\small $r_L^{-1}$} [b] at 203 343
\endlabellist
\centering
\includegraphics[scale=0.75]{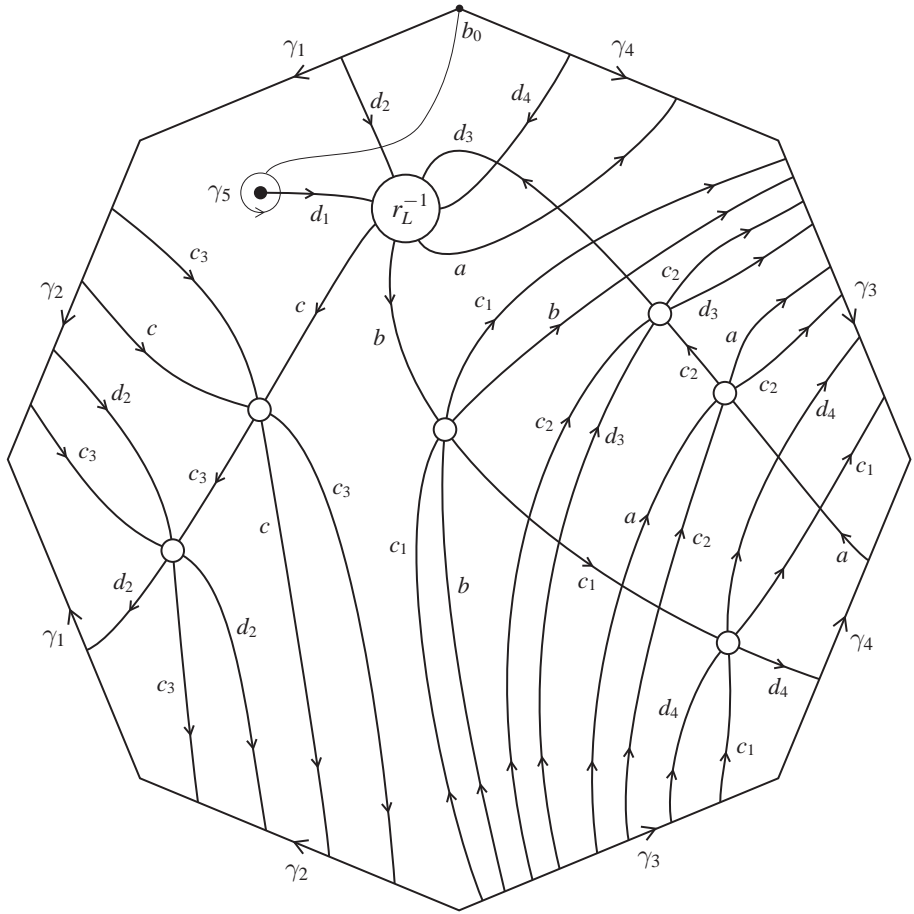}
\caption{Chart for Lefschetz fibration of Korkmaz and Ozbagci}
\label{exampleB}
\end{figure}

\begin{figure}[ht!]
\labellist
\footnotesize \hair 2pt
\pinlabel $x$ [b] at 4 56
\pinlabel $y$ [b] at 26 56
\pinlabel $x$ [b] at 49 56
\pinlabel $y$ [t] at 4 0
\pinlabel $x$ [t] at 26 0
\pinlabel $y$ [t] at 49 0
\pinlabel $:=$ [b] at 70 22
\pinlabel $x$ [b] at 91 56
\pinlabel $y$ [b] at 113 56
\pinlabel $x$ [b] at 135 56
\pinlabel $y$ [t] at 91 0
\pinlabel $x$ [t] at 113 0
\pinlabel $y$ [t] at 135 0
\pinlabel $z$ [b] at 117 29
\pinlabel {(a)} [t] at 70 -15
\pinlabel $x$ [br] at 180 44
\pinlabel $y$ [bl] at 210 44
\pinlabel $y$ [tr] at 180 12
\pinlabel $x$ [tl] at 210 12
\pinlabel $:=$ [b] at 228 22
\pinlabel $x$ [br] at 243 44
\pinlabel $y$ [bl] at 273 44
\pinlabel $y$ [tr] at 243 12
\pinlabel $x$ [tl] at 273 12
\pinlabel {(b)} [t] at 228 -15
\endlabellist
\centering
\includegraphics[scale=0.8]{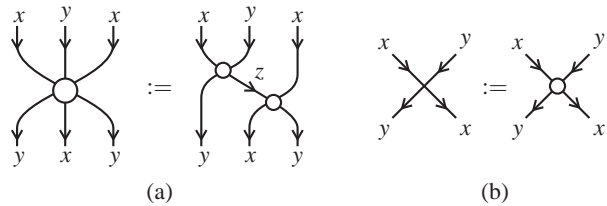}
\vspace{5mm}
\caption{Simplification of vertices}
\label{verticesE}
\end{figure}

\begin{prob}
Study various properties of Lefschetz fibrations by using chart description for the triple 
$(\mathcal{X},\mathcal{R},\mathcal{S})$ defined above. 
\end{prob}

We next mention chart description for Lefschetz fibrations with bordered base and fiber. 
Kamada \cite{Kamada2007} gave a general theory for charts in a compact oriented surface 
with boundary. 
Various presentations of mapping class groups of surfaces with boundary 
have been investigated by 
researchers including Gervais \cite{Gervais2001}, 
Labru\`{e}re and Paris \cite{LP2001}, Margalit and McCammond \cite{MM2009}. 
Combining these two kinds of studies, one can immediately obtain a chart 
description for Lefschetz fibrations with bordered base and fiber. 

\begin{prob}
Make use of chart description to study PALFs and Stein surfaces. 
\end{prob}

It would be worth considering compositions of monodromy representations with 
appropriate homomorphisms and charts corresponding to the compositions. 
For example, Hasegawa \cite{Hasegawa2006a, Hasegawa2006} adopted a homomorphism 
from the $m$--string braid group $B_m$ to the semi-direct product 
$(\mathbb{Z}_2)^m\times S_m$, while Endo and Kamada \cite{EK2014} used 
a standard epimorphism from the hyperelliptic mapping class group of 
a closed oriented surface of genus $g$ to 
the mapping class group of a sphere with $2g+2$ marked points. 

\begin{prob}
Consider chart descriptions for `nice' representations of mapping class groups to 
study invariants and classifications of Lefschetz fibrations. 
\end{prob}

Theorem \ref{main1} and Theorem \ref{main2} tell us that 
the numbers of singular fibers of all types and the signature of the total space 
completely determine the stable isomorphism class of a Lefschetz fibration 
with given base and fiber. 
Thus any numerical invariant of Lefschetz fibrations which is additive under 
fiber sum is determined by these invariants in principle. 

\begin{prob}
Construct numerical invariants of Lefschetz fibrations 
which are {\it not} additive under fiber sum. 
\end{prob}

Nosaka \cite{Nosaka2014} has recently defined an invariant which is not additive 
under fiber sum. 
Non-numerical invariants such as monodromy group would be also useful 
(see Matsumoto \cite{Matsumoto1996} and Park and Yun \cite{PY2009, PY2011}).

{\bf Acknowledgements} \ \ 
The authors would like to thank the referees for their helpful suggestions and corrections. 
The first author was partially supported by JSPS KAKENHI Grant Numbers 
21540079, 25400082. 
The third author was partially supported by JSPS KAKENHI Grant Numbers 
21340015, 26287013. 
The fourth author was partially supported by JSPS KAKENHI Grant Numbers 
21740042, 26400082.

%
%
%
%

\end{document}